\documentclass[12pt]{amsart}
\usepackage[arrow, matrix, curve]{xy}
\usepackage{amssymb,amsmath,amsfonts, mathrsfs} 
\usepackage{amsthm} 
\newcommand{\IC}{\mathbb{C}}
\newcommand{\IR}{\mathbb{R}}

\newcommand{\IDD}{\mathscr{D}}

\newcommand{\IN}{\mathbb{N}}
\newcommand{\IZ}{\mathbb{Z}}

\newcommand*{\longhookrightarrow}%
               {\ensuremath{\lhook\joinrel\relbar\joinrel\rightarrow}}

\newcommand{\Id}{{\rm d}}

\newcommand{\f}{\frac}

\theoremstyle{plain}            
\newtheorem{theorem}{theorem}[section]
\newtheorem{Lemma}[theorem]{Lemma}
\newtheorem{Corollary}[theorem]{Corollary}
\newtheorem{Theorem}[theorem]{Theorem}
\newtheorem{Proposition}[theorem]{Proposition}
\newtheorem{Propandef}[theorem]{Proposition and definition}

\theoremstyle{definition}       
\newtheorem{Definition}[theorem]{Definition}

\newtheorem{Remark}[theorem]{Remark}
\newtheorem{Example}[theorem]{Example}

\begin{document}

\title[$H=W$ on arbitrary manifolds]{$\mathsf{L}^1$-elliptic regularity and $H=W$ on the whole 
$\mathsf{L}^p$-scale on arbitrary manifolds}

\author{Davide Guidetti, Batu G\"uneysu {\rm and} Diego Pallara}
\date{\today}
\address{\newline
Davide Guidetti, {\tt davide.guidetti@unibo.it}\newline
         \indent {\rm  Dipartimento di Matematica, Universit\`a di Bologna, Bologna, Italy}
\newline\newline
Batu G\"uneysu, {\tt gueneysu@math.hu-berlin.de} \newline
\indent {\rm Institut f\"ur Mathematik,
        Humboldt-Universit\"at zu Berlin, Berlin, Germany}  \newline  \newline
Diego Pallara, {\tt diego.pallara@unisalento.it}\newline
         \indent {\rm Dipartimento di Matematica e Fisica \emph{Ennio De Giorgi}, Universit\`a del Salento, 
         Lecce, Italy}
}

\begin{abstract} We define abstract Sobolev type spaces on $\mathsf{L}^p$-scales, $p\in [1,\infty)$, 
on Hermitian vector  bundles over possibly noncompact manifolds, which are induced by smooth measures 
and families $\mathfrak{P}$ of  linear partial  differential operators, and we prove the density of the 
corresponding smooth Sobolev sections in these spaces under a generalized ellipticity condition on the 
underlying family. In particular, this implies a covariant version of Meyers-Serrin\rq{}s theorem on the 
whole $\mathsf{L}^p$-scale, for arbitrary Riemannian manifolds. Furthermore, we prove a new local 
elliptic regularity result in $\mathsf{L}^1$ on the Besov scale, which shows that the above generalized 
ellipticity condition is satisfied on the whole $\mathsf{L}^p$-scale, if some differential operator from 
$\mathfrak{P}$ that has a sufficiently high (but not necessarily the highest) order is elliptic.
\end{abstract}

\maketitle
 
\setcounter{page}{1}

\section{Introduction}

Let us recall that a classical result of Meyers and Serrin \cite{meyers} states that for any open 
subset $U$ of the Euclidean $\IR^m$ and any $k\in\IN_{\geq 0}$, $ p\in [1,\infty)$, one has 
$\mathsf{W}^{k,p}(U)=\mathsf{H}^{k,p}(U)$, where $\mathsf{W}^{k,p}(U)$ is given as
the complex Banach 
space of all $f\in\mathsf{L}^1_{\mathrm{loc}}(U)$ such that 
\begin{align}\label{ms}
\left\|f\right\|_{k,p}:=
\Big(\int_{U}|f(x)|^p\Id x\Big)^{1/p}+
\sum_{|\alpha|\leq k}\Big(\int_{U}|\partial^{\alpha}f(x)|^p\Id x\Big)^{1/p}<\infty,
\end{align}
and where $\mathsf{H}^{k,p}(U)$ is defined as the closure of
$\mathsf{W}^{k,p}(U)\cap\mathsf{C}^{\infty}(U)$ 
with respect to the norm $\left\|\bullet\right\|_{k,p}$. \\
On the other hand, thinking for example of Riemannian geometry on noncompact manifolds, it becomes 
very natural to ask under what minimal assumptions one can replace the partial derivatives in (\ref{ms}) 
by more general partial differential operators, that are nonelliptic and typically vector-valued. In fact, in 
order to deal with all possible geometric situations simultaneously, we introduce an abstract notion of a 
$\mathfrak{P}$-Sobolev space $\Gamma_{\mathsf{W}^{\mathfrak{P},p}_{\mu}}(X,E)$ of 
$\mathsf{L}^p_{\mu}$-sections in a Hermitian vector bundle $E\to X$ (cf. Definition \ref{asa}). 
Here, $X$ is a possibly noncompact manifold, $\mu$ is a smooth measure on $X$ (which may, but 
need not come from a Riemannian metric in general), $F_1,\dots,F_s\to X$ are  Hermitian vector 
bundles, and the datum $\mathfrak{P}=\{P_1,\dots,P_s\}$ is a finite collection such that each 
$P_j$ is a linear partial differential operator of order $\leq k_j$ from $E$ to $F_j$. With 
$\left\|\bullet\right\|_{\mathfrak{P},p,\mu}$ the canonical norm on 
$\Gamma_{\mathsf{W}^{\mathfrak{P},p}_{\mu}}(X,E)$, the question we address here is: \\
\emph{Under which assumptions on $\mathfrak{P}$ is the space of smooth Sobolev sections}
\begin{align}\label{frage}
 \Gamma_{\mathsf{C}^{\infty}}(X,E)\cap \Gamma_{\mathsf{W}^{\mathfrak{P},p}_{\mu}}(X,E)\>
\text{\em dense in } \Gamma_{\mathsf{W}^{\mathfrak{P},p}_{\mu}}(X,E)\>
\text{\em w.r.t. $\left\|\bullet\right\|_{\mathfrak{P},p,\mu}$?}
\end{align}
To this end, the highest differential order $k:=\max\{k_1,\dots,k_s\}$ of the system $\mathfrak{P}$, 
plays an essential role: Namely, it turns out that even on an entirely local level (cf. Lemma \ref{molli}), 
the machinery of Friedrichs mollifiers precisely applies
\begin{align}
\text{ either if $k<2$, or if } \Gamma_{\mathsf{W}^{\mathfrak{P},p}_{\mu}}(X,E)  
\subset \Gamma_{\mathsf{W}^{k-1,p}_{\mathrm{loc}}}(X,E).\label{deds}
\end{align}
With this observation, our basic abstract result \emph{Theorem \ref{main} precisely states that the 
local regularity (\ref{deds}) implies (\ref{frage}), and that furthermore any compactly supported 
element of $\Gamma_{\mathsf{W}^{\mathfrak{P},p}_{\mu}}(X,E)$ can be even approximated by a sequence 
from $\Gamma_{\mathsf{C}^{\infty}_{\mathrm{c}}}(X,E)$.}\\
This result turns out to be optimal in the following sense (cf. Example \ref{bep}): There are differential 
operators $P$ such that for any $q>1$ one has 
\begin{align*}
&\mathsf{W}^{P,q}\subset \mathsf{W}^{\mathrm{ord}(P)-2,q}_{\mathrm{loc}}, 
\>\>\> \mathsf{W}^{P,q}\not\subset \mathsf{W}^{\mathrm{ord}(P)-1,q}_{\mathrm{loc}},\\
&\mathsf{C}^{\infty} \cap \mathsf{W}^{P,q} \text{ is not dense in } \mathsf{W}^{P,q}.
\end{align*}
Thus it remains to examine the regularity assumption (\ref{deds}) in applications, where of course we 
can assume $k\geq 2$. \\
To this end, it is clear from classical local elliptic estimates that for $p>1$, (\ref{deds}) is 
satisfied whenever
there is some elliptic $P_j$ with $k_j\geq k-1$. However, the $\mathsf{L}^1$-case $p=1$ is much more subtle, 
since the usual local elliptic regularity is well-known to fail here (cf. Remark \ref{counte}). However, 
\emph{in Theorem \ref{regu} we prove a new modified local elliptic regularity result on the scale of 
Besov spaces, which implies that in the $\mathsf{L}^1$-situation, one loses exactly one differential 
order of regularity when compared with the usual local elliptic $\mathsf{L}^p$, $p>1$, estimates.} This 
in turn shows that for $p=1$, (\ref{deds}) is satisfied whenever there is some elliptic $P_j$ with 
$k_j= k$. These observations are collected in Corollary \ref{ell}. The proof of Theorem \ref{regu} 
relies on \emph{a new existence and uniqueness result, (cf. Proposition \ref{pr2} in Section 
\ref{beweis2}) for certain systems of linear elliptic of PDE\rq{}s on the Besov scale,} which 
is certainly also of an independent interest.  \\
Finally, we would like to point out that the regularity (\ref{deds}) does not require the ellipticity 
of any $P_j$ at all. Indeed, \emph{in Corollary \ref{meyers} we prove that if $(M,g)$ is a possibly 
noncompact Riemannian manifold and $E\to M$ a Hermitian vector bundle with a (not necessarily Hermitian) 
covariant derivative $\nabla$, then for any $s\in\IN$ and $p\in (1,\infty)$, the Sobolev space}
\[
\Gamma_{\mathsf{W}^{s,p}_{\nabla,g}}(M,E):=
\Gamma_{\{\nabla^{1}_g,\dots,\nabla^{s}_g\},\mathsf{L}^p_{\mathrm{vol}_g}}(M,E).
\]
\emph{satisfies }
$$
\Gamma_{\mathsf{W}^{s,p}_{\nabla,g}}(M,E)\subset\Gamma_{\mathsf{W}^{s,p}_{\mathrm{loc}}}(M,E),
$$
which means that we do not even have to use the full strenght of Theorem \ref{main} here. To the best 
of our knowledge, the resulting density of 
$$
\Gamma_{\mathsf{W}^{s,p}_{\nabla,g}}(M,E)\cap \Gamma_{\mathsf{C}^{\infty}}(M,E)
\text{ in } \Gamma_{\mathsf{W}^{s,p}_{\nabla,g}}(M,E)
$$
is entirely new in this generality.

\section{Main results}

Throughout, \emph{let $X$ be a smooth $m$-manifold (without boundary)} which is allowed to be noncompact. 
For subsets $Y_1,Y_2\subset X$ we write 
\[
Y_1\Subset Y_2, \text{ if and only if $Y_1$ is open, 
$\overline{Y_1}\subset Y_2$, and $\overline{Y_1}$ is compact.}
\]
We abbreviate that for any $k\in\IN_{\geq 0}$, we denote with $\IN^m_{k}$ the set of multi-indices 
$\alpha\in (\IN_{\geq 0})^m$  with $|\alpha|:=\sum^{m}_{j=1}\alpha_j \leq k$. Note that 
$(0,\dots,0)\in \IN^m_k$ by definition, for any $k$. \\
In order to be able to deal with Banach structures that are not necessarily induced by Riemannian 
structures \cite{Br}, \emph{we fix a smooth measure $\mu$ on $X$,} that is, $\mu$ is a Borel measure on 
$X$ such that for any chart $U$ for $X$ there is a (necessarily unique) 
$0<\mu_U\in\mathsf{C}^{\infty}(U)$ with the property that   
\[
\mu(A)=\int_A \mu_U(x^1,\cdots, x^m)\Id x^1\cdots\Id x^m\>\text{ for all Borel sets $A\subset U$}, 
\]
where $\Id x=\Id x^1\cdots\Id x^m$ stands for Lebesgue integration.\\
We always understand our linear spaces to be complex-valued, and an index \lq\lq{}$\mathrm{c}$\rq\rq{} 
in spaces of sections or functions stands for \lq\lq{}compact support\rq\rq{}, where in the context of 
equivalence classes (with respect to some/all $\mu$ as above) of Borel measurable sections, compact 
support of course means \lq\lq{}compact essential support\rq\rq{}. \\
Assume for the moment that we are given smooth complex vector bundles $E\to X$, $F\to X$, with 
$\mathrm{rank}(E)=\ell_0$ and $\mathrm{rank}(F)=\ell_1$. The linear space of smooth sections in 
$E\to X $ is denoted by $\Gamma_{\mathsf{C}^{\infty}}(X,E)$, and the linear space of equivalence 
classes of Borel sections in $E\to X$ is simply written as $\Gamma(X,E)$.\\
We continue by listing some conventions and some notation concerning linear differential operators 
and distributions on manifolds. We start by adding the following two classical definitions on linear 
differential for the convenience of the reader, who can find these and the corresponding basics in 
\cite{nico,wald,chaz,lawson}. We also refer the reader to \cite{batu} (and the references therein) 
for the jet bundle aspects of (possibly nonlinear) partial differential operators.

\begin{Definition}\label{ops} A morphism of linear sheaves
$$
P: \Gamma_{\mathsf{C}^{\infty}}(X,E)\longrightarrow \Gamma_{\mathsf{C}^{\infty}}(X,F)
$$
is called a \emph{smooth linear partial differential operator of order at most $k$}, if for any chart
$$
x=(x^1,\dots,x^m):U\longrightarrow \IR^m 
$$
for $X$ which admits local frames $e_1,\dots,e_{\ell_0}\in\Gamma_{\mathsf{C}^{\infty}}(U,E)$, 
$f_1,\dots,f_{\ell_1}\in\Gamma_{\mathsf{C}^{\infty}}(U,F)$, and any $\alpha\in\IN^m_k$, there 
are (necessarily uniquely determined) smooth functions
$$
P_{\alpha}:U\longrightarrow \mathrm{Mat}(\IC;\ell_0\times \ell_1)
$$
such that for all $(\phi^1,\dots,\phi^{\ell_0})\in\mathsf{C}^{\infty}(U,\IC^{\ell_0})$ one has
$$
P\sum^{\ell_0}_{i=1}\phi^i e_i=\sum^{\ell_1}_{j=1}\sum^{\ell_0}_{i=1}\sum_{\alpha\in\IN^m_k}
P_{\alpha ij}\frac{\partial^{|\alpha|}\phi^{i}}{\partial x^{\alpha}}f_j\>\>\text{ in $U$}.
$$
\end{Definition}

The linear space of smooth at most $k$-th order linear partial differential operators is denoted by 
$\mathscr{D}^{(k)}_{\mathsf{C}^{\infty}}(X;E,F)$.

\begin{Propandef} Let $P\in \mathscr{D}^{(k)}_{\mathsf{C}^{\infty}}(X;E,F)$.\\
\emph{a)} The (linear principal) symbol of $P$
is the unique morphism of smooth complex vector bundles over $X$,
\begin{align*}
\sigma_P: (\mathrm{T}^*X)^{\odot k}\longrightarrow \mathrm{Hom}(E,F),
\end{align*}
where $\odot$ stands for the symmetric tensor product, such that for all $x:U\to \IR^m$, 
$e_1,\dots,e_{\ell_0}$, $f_1,\dots,f_{\ell_1}$, $\alpha$ as in Definition \ref{ops} one has
$$
\sigma_{P}\left(\Id x^{ \odot\alpha} \right)e_i= \sum^{\ell_1}_{j=1}P_{\alpha i j}f_j\>\>\text{ in $U$}.
$$
\emph{b)} $P$ is called \emph{elliptic}, if for 
all $x\in X$, $v\in\mathrm{T}^*_x X\setminus \{0\}$, the linear map 
\[
\sigma_{P,x}(v ):=\sigma_{P,x}(v^{\otimes k}): E_x\longrightarrow  F_x
\> \text{ is in $\mathrm{GL}(E_x,F_x)$.}
\]
\end{Propandef}

We recall that the linear space $\Gamma_{\mathsf{W}^{k,p}_{\mathrm{loc}}}(X,E)$ of 
\emph{local $\mathsf{L}^{p}$-Sobolev sections in $E\to X$ with differential order $k$} is defined 
to be the space of $f\in\Gamma(X,E)$ such that for all charts $U\subset X$ which admit a local 
frame $e_1,\dots,e_{\ell_0}\in\Gamma_{\mathsf{C}^{\infty}}(U,E)$, one has 
$$
(f^1,\dots,f^{{\ell_0}}) \in \mathsf{W}^{k,p}_{\mathrm{loc}}(U,\IC^{\ell_0}),\>\>
\text { if $f=\sum^{\ell_0}_{j=1}f^je_j$ in $U$.} 
$$
In particular, we have the space of locally $p$-integrable sections
$$
\Gamma_{\mathsf{L}^p_{\mathrm{loc}}}(X,E):=\Gamma_{\mathsf{W}^{0,p}_{\mathrm{loc}}}(X,E).
$$ 
The linear space of \emph{distributional sections} in $E\to X$ is defined by
\begin{align*}
&\Gamma_{\mathsf{D}\rq{}}(X,E):=\text{ topological dual of $\Gamma_{\mathsf{D}}(X,E)$, where}\\
&\Gamma_{\mathsf{D}}(X,E):=\Gamma_{\mathsf{C}^{\infty}_{\mathrm{c}}}(X,E^*\otimes |X|),
\end{align*}
and where $|X|\to X$ denotes the bundle of $1$-densities, which is a smooth complex line bundle. 
We have the canonical embedding
$$
\Gamma_{\mathsf{L}^1_{\mathrm{loc}}}(X,E)\longhookrightarrow \Gamma_{\mathsf{D}\rq{}}(X,E),
$$
given by identifying $f\in\Gamma_{\mathsf{L}^1_{\mathrm{loc}}}(X,E)$ with the distribution
$$
\left\langle f,\Psi\right\rangle := \int_X \Psi[f],\>\>\>\Psi\in\Gamma_{\mathsf{D}}(X,E). 
$$
We continue with (cf. Proposition 1.2.12 in \cite{wald}, or \cite{chaz}):

\begin{Lemma} For any $P\in\mathscr{D}^{(k)}_{\mathsf{C}^{\infty}}(X;E,F)$, there is a unique 
differential operator
$$
P^t\in\mathscr{D}^{(k)}_{\mathsf{C}^{\infty}}(X;F^*\otimes |X|,E^*\otimes |X|),
$$
\emph{the transpose of $P$}, which satisfies
\begin{align*} 
&\int_X P^t\Psi[\phi]=\int_X \Psi[P\phi] \>\>
\text{ for all $\Psi \in\Gamma_{\mathsf{C}^{\infty}}(X,F^*\otimes |X|)$,}
\\ \nonumber
& \text{  and all $\phi\in \Gamma_{\mathsf{C}^{\infty}}(X,E)$, with either $\phi$ or $\Psi$ 
compactly supported.} 
\end{align*}
\end{Lemma}

Using the transpose, one extends any $P\in \mathscr{D}^{(k)}_{\mathsf{C}^{\infty}}(X;E,F)$ 
canonically to a linear map 
$$
P: \Gamma_{\mathsf{D}\rq{}}(X,E)\longrightarrow \Gamma_{\mathsf{D}\rq{}}(X,F), 
$$
by requiring
$$
\left\langle Ph,\phi\right\rangle  = \left\langle h,P^t\phi\right\rangle\>
\text{ for all $h\in \Gamma_{\mathsf{D}\rq{}}(X,E)$, $\phi\in \Gamma_{\mathsf{D}}(X,F)$}.
$$

\begin{Remark}\label{adjo}1. Assume that $E\to X$ and $F\to X$ come equipped with smooth Hermitian structures $h_E(\bullet,\bullet)$ and $h_F(\bullet,\bullet)$, respectively. We define $P^{\mu}\in\mathscr{D}^{(k)}_{\mathsf{C}^{\infty}}(X;F^*,E^*)$ by 
$$
(P^{\mu}\psi )\otimes \mu := P^{t}(\psi\otimes \mu),\>\> \psi \in\Gamma_{\mathsf{C}^{\infty}}(X,F^*),
$$
and $P^{\mu, h_E,h_F}\in\mathscr{D}^{(k)}_{\mathsf{C}^{\infty}}(X;F,E)$ by the diagram
\[
\begin{xy}
\xymatrix{
\Gamma_{\mathsf{C}^{\infty}}(X,F^*)\>\>\ar[rr]^{P^{\mu}} & & \>\>\Gamma_{\mathsf{C}^{\infty}}(X,E^*)\ar[dd]^{\tilde{h}_E^{-1}} \\	 \\
\Gamma_{\mathsf{C}^{\infty}}(X,F)\ar[uu]^{\tilde{h}_F}\ar@{.>}_{P^{\mu, h_E,h_F}}[rr] & & \Gamma_{\mathsf{C}^{\infty}}(X,E)
}
\end{xy}
\]
where $\tilde{h_E}$ and $\tilde{h_F}$ stand for the isomorphisms of $\mathsf{C}^{\infty}(X)$-modules which are induced by $h_E$ and $h_F$, respectively. Then $P^{\mu, h_E,h_F}$ is the uniquely determined element of $\mathscr{D}^{(k)}_{\mathsf{C}^{\infty}}(X;F,E)$ which satisfies
\begin{align*}
&\int_X h_E\left(P^{\mu, h_E,h_F} \psi,\phi\right)\Id  \mu = \int_X h_F\left(\psi , P\phi\right) \Id \mu 
\end{align*}
for all $\psi \in\Gamma_{\mathsf{C}^{\infty}}(X,F)$, $\phi\in \Gamma_{\mathsf{C}^{\infty}}(X,E)$ with either $\phi$ or $\psi$ compactly supported.\\
2. Given $f_1\in \Gamma_{\mathsf{L}^{1}_{\mathrm{loc}}}(X,E)$, $f_2\in \Gamma_{\mathsf{L}^{1}_{\mathrm{loc}}}(X,F)$ one has $Pf_1=f_2$, if and only if for \emph{some} triple $(\mu,h_E,h_F)$ as above it holds that 
\begin{align}
&\int_X h_E\left(P^{\mu, h_E,h_F} \psi,f_1\right)\Id  \mu = \int_X h_F\left(\psi , f_2\right) \Id \mu\>\text{ for all $\psi \in\Gamma_{\mathsf{C}^{\infty}_{\mathrm{c}}}(X,F)$ },\label{sdf}
\end{align}
and then (\ref{sdf}) automatically holds for \emph{all} such triples $(\mu,h_E,h_F)$.
\end{Remark}

From now on, given a smooth \emph{Hermitian} vector bundle $E\to X$ and $p\in [1,\infty]$, abusing the 
notation as usual, $(\bullet,\bullet)_x$ denotes the inner product on the fiber $E_x$, with 
$\left|\bullet\right|_x$ the corresponding norm, and we get a Banach space
\[
\Gamma_{\mathsf{L}^p_{\mu}}(X,E):=
\left\{f\left|f\in \Gamma(X,E),\left\| f\right\|_{p,\mu}<\infty\right\}\right.,
\]
where
\[
\left\| f\right\|_{p,\mu}:=\begin{cases}&\Big(\int_X \big|f(x)\big|^p_x  \mu(\Id x)\Big)^{1/p} 
,\text{ if $p<\infty$} \\
&\inf\{C|C\geq 0, |f|\leq C\text{ $\mu$-a.e.}\},\text{ if $p=\infty$.}\end{cases}
\]
Of course, $\Gamma_{\mathsf{L}^2_{\mu}}(X,E)$ becomes a Hilbert space with its canonical inner 
product. \\
The following definition is in the center of this paper:

\begin{Definition}\label{asa} 
Let $p\in [1,\infty]$, $s\in\IN$, $k_1\dots,k_s\in\IN_{\geq 0}$, and for each $i\in\{1,\dots,s\}$ 
let $E\to X$, $F_i\to X$ be smooth Hermitian vector bundles and let $\mathfrak{P}:=\{P_1,\dots,P_s\}$ 
with $P_{i}\in\allowbreak \mathscr{D}_{\mathsf{C}^{\infty}}^{(k_i)}(X;E,F_i)$. Then the Banach space 
\begin{align*}
&\Gamma_{\mathsf{W}^{\mathfrak{P},p}_{\mu}}(X,E)\\
&:=\left.\Big\{f\right|f\in\Gamma_{\mathsf{L}^{p}_{\mu}}(X,E), 
P_{i} f\in\Gamma_{ \mathsf{L}^{p}_{\mu} }(X,F_i)\text{ for all $i\in\{1,\dots,s\}$}\Big\}\\
&\>\>\>\> \>\>\>\>\subset \Gamma_{\mathsf{L}^{p}_{\mu}}(X,E),\>
\text{ with norm $\left\| f\right\|_{\mathfrak{P},p,\mu}
:=\left(\left\|f\right\|^p_{p,\mu}+\sum^s_{i=1}\left\|P_{i} f\right\|^p_{p,\mu}\right)^{1/p}$},
\end{align*}
is called the \emph{$\mathfrak{P}$-Sobolev space of $\mathsf{L}^{p}_{\mu}$-sections} in $E\to X$. 
\end{Definition}

Note that in the above situation, $\Gamma_{\mathsf{W}^{\mathfrak{P},2}_{\mu}}(X,E)$ is a Hilbert space 
with the obvious inner product, and we have the linear space
\begin{align*}
&\Gamma_{\mathsf{W}^{\mathfrak{P},p}_{\mathrm{loc}}}(X,E)\\
&:=
\left\{f\left|f\in\Gamma_{\mathsf{L}^p_{\mathrm{loc}}}(X,E),P_{i} 
f\in\Gamma_{ \mathsf{L}^{p}_{\mathrm{loc}} }(X,F_i)\text{ for all $i\in\{1,\dots,s\}$} 
\right\}\right. 
\end{align*}
of \emph{locally $p$-integrable sections in $E\to X$ with differential structure $\mathfrak{P}$,} 
which of course does not depend on any Hermitian structures.\\
In this context, let us record the following local elliptic regularity result, whose 
$\mathsf{L}^{p}_{\mathrm{loc}}$-case, $p\in (1,\infty)$, is classical (see for example Theorem 10.3.6 in 
\cite{nico}), while the $\mathsf{L}^{1}_{\mathrm{loc}}$-case 
seems to be entirely new, and can be considered as our first main result:

\begin{Theorem}\label{regu} Let $U\subset \IR^m$ be open, let $k\in\IN_{\geq 0}$, $\ell\in\IN$, and 
let $P\in\IDD^{(k)}_{\mathsf{C}^{\infty}}(U;\IC^{\ell},\IC^{\ell})$,
$$
P= \sum_{\alpha\in \IN^m_{k} }  P_{\alpha}\partial^{\alpha},
\>\text{ with $P_{\alpha}:U\longrightarrow\mathrm{Mat}(\IC;\ell\times \ell)$ 
in $\mathsf{C}^{\infty}$}
$$
be elliptic. Then the following results hold true:\\
\emph{a)} If $p\in (1,\infty)$, then for any $f\in \mathsf{L}^{p}_{\mathrm{loc}}(U,\IC^{\ell})$ with 
$Pf\in \mathsf{L}^{p}_{\mathrm{loc}}(U,\IC^{\ell})$ one has 
$f\in\mathsf{W}^{k,p}_{\mathrm{loc}}(U,\IC^{\ell})$.\\
\emph{b)} For any $f\in\mathsf{L}^{1}_{\mathrm{loc}}(U,\IC^{\ell})$ with 
$Pf\in \mathsf{L}^{1}_{\mathrm{loc}}(U,\IC^{\ell})$ it holds that 
$f\in\mathsf{W}^{k-1,1}_{\mathrm{loc}}(U,\IC^{\ell})$.
\end{Theorem}

Before we come to the proof, a few remarks are in order:

\begin{Remark}\label{counte} 
In fact, we are going to prove the following much stronger 
statement in part b): Under the assumptions of Theorem \ref{regu} b), for any 
$f\in\mathsf{L}^{1}_{\mathrm{loc}}(U,\IC^{\ell})$ with $Pf\in \mathsf{L}^{1}_{\mathrm{loc}}(U,\IC^{\ell})$, 
one has that for any $\psi\in\mathsf{C}^{\infty}_{\mathrm{c}}(U)$, the distribution $\psi f$ is in the 
Besov space 
$$
\mathsf{B}^{k}_{1,\infty}(\IR^m,\IC^{\ell})\subset\mathsf{W}^{k-1,1}(\IR^m,\IC^{\ell}).
$$
This in turn is proved using a new existence and uniqueness result (cf. Proposition \ref{pr2} in Section 
\ref{beweis2}) for certain systems of linear elliptic PDE\rq{}s on the Besov scale. We refer the reader to 
Section \ref{beweis2} for the definition and essential properties of the Besov spaces 
$\mathsf{B}^{\beta}_{p,q}(\IR^m,\IC^\ell)\subset\mathsf{S}\rq{}(\IR^m,\IC^\ell)$ (with 
$\mathsf{S}\rq{}(\IR^m)$ the Schwartz distributions), where $\beta\in\IR$, $p,q\in [1,\infty]$. Note that 
in the situation of Theorem \ref{regu} b), the assumptions 
$f,Pf\in\mathsf{L}^{1}_{\mathrm{loc}}(U,\IC^{\ell})$, 
do not imply $f\in\mathsf{W}^{k,1}_{\mathrm{loc}}(U,\IC^{\ell})$: An explicit counter example has 
been given in \cite{rother} for the Euclidean Laplace operator. In fact, it follows from results of 
\cite{davide2} that for any strongly elliptic differential operator $P$ in $\IR^m$ with constant 
coefficients and order $2k$, there is a $f$ with $f,Pf\in\mathsf{L}^{1}_{\mathrm{loc}}(\IR^m)$, and 
$f\notin\mathsf{W}^{2k,1}_{\mathrm{loc}}(\IR^m)$. 
In this sense, the above $k$-th order Besov regularity can be considered to be optimal.
\end{Remark}

\begin{proof}[Proof of Theorem \ref{regu} b)] In this proof, we denote with $(\bullet,\bullet)$ the 
standard inner product in each $\IC^n$, and with $\left|\bullet\right|$ the corresponding norm and 
operator norm, and $\mathrm{B}_r(x)$ stands for the corresponding open ball of radius $r$ around $x$. 
Let us consider the formally self-adjoint elliptic partial differential operator 
$$
T:=P^{\dagger}P=
\sum_{\alpha \in \IN_{2k}^m}T_\alpha \partial^\alpha\in \mathscr{D}_{\mathsf{C}^\infty}^{(2k)}
(\IR^m; \IC^\ell, \IC^\ell).
$$
Here, $P^{\dagger}\in\IDD^{(k)}_{\mathsf{C}^{\infty}}(U;\IC^{\ell},\IC^{\ell})$ denotes the usual 
formal adjoint of $P$, which is well-defined by
$$
\int_{U} (P^{\dagger}\varphi_1,\varphi_2) \Id x=\int_U ( \varphi_1,P\varphi_2)\Id x,
$$
for all $\varphi_1,\varphi_2\in \mathsf{C}^{\infty}(U,\IC^{\ell})$ one of which having a compact 
support, in other words, $P^{\dagger}$ is nothing but the operator $P^{\mu,h_E,h_F}$ from Remark \ref{adjo}.1, with respect to the Lebesgue measure and the canonical Hermitian structures on the trivial bundles. By a standard partition of unity argument, it suffices to prove that if 
$\psi\in\mathsf{C}^{\infty}_{\mathrm{c}}(U)$ with 
\begin{align}\label{incl} 
\mathrm{supp}(\psi)\subset \mathrm{B}_{t_0}(x_0)\subset U 
\end{align}
for some $x_0\in U,\ t_0>0$ we have $\psi f\in\mathsf{B}^{k}_{1,\infty}(\IR^m,\IC^{\ell})$. 
The proof consists of two steps: 
We first construct a differential operator $Q^{\psi}$ which satisfies the assumptions of 
Proposition \ref{pr2}, and which coincides with $T$ near $\mathrm{supp}(\psi)$, and then we apply 
Proposition \ref{pr2} together with a maximality argument to $Q^{\psi}$ to deduce the thesis.\\
we can assume that there are $t_0>0$, $x_0\in U$ such that 
We also take some $\phi\in\mathsf{C}^{\infty}_{\mathrm{c}}(U)$ with $\phi=1$ on $\mathrm{B}_{t_0}(x_0)$, 
and for any $0<t<t_0$ we set
\[
C_{t}:=\max_{y\in \overline{\mathrm{B}_t(x_0)},\alpha\in \IN_{2k}^m} 
|T_{\alpha i  j}(y)-T_{\alpha i  j}(x_0)|,
\]
and we pick a $\chi_t\in\mathsf{C}^{\infty}_{\mathrm{c}}(\IR^2,\IR^2)$ with $\chi_t(z)=z$ for all $z$ 
with $|z|\leq C_t$, and $|\chi_t(z)|\leq 2C_t$ for all $z$. We define a differential operator
\begin{align*}
&Q^{(t)}=\sum_{\alpha \in \IN_{2k}^m}Q^{(t)}_{\alpha} 
\partial^\alpha\in \mathscr{D}_{\mathsf{C}^\infty}^{(2k)}(\IR^m; \IC^\ell, \IC^\ell),
\\
&Q^{(t)}_{\alpha i j}(x):= T_{\alpha i j}(x_0)+
\chi_t\big(\phi(x)(T_{\alpha i j}(x)-T_{\alpha i j}(x_0))\big)
\\
&\>\>\>\>\>\>=:T_{\alpha i j}(x_0)+A^{(t)}_{\alpha i j}(x)
\end{align*}
(with the usual extension of $\phi (T_{\alpha i j}-T_{\alpha i j}(x_0))$ to zero away from $U$ 
being understood, so in particular we have $Q^{(t)}_{\alpha i j}(x)= T_{\alpha i j}(x_0)$, if 
$x\in \IR^m\setminus U$). Let $\zeta\in\IR^m\setminus\{0\}$, $\eta\in\IC^\ell$ be arbitrary. 
Then using $\sigma_{T,x_0}=\sigma_{P,x_0}^{\dagger}\sigma_{P,x_0}$, and that 
$$
\IR^m\setminus\{0\}\ni \zeta\rq{}\longmapsto  \sigma_{P,x_0}(\mathrm{i}\zeta\rq{})=
\sum_{\alpha\in\IN^m_{k}, |\alpha|=k}P_{\alpha}(x_0)(\mathrm{i}\zeta\rq{})^{\alpha}\in 
\mathrm{GL}(\IC;\ell\times \ell)
$$
is well-defined and positively homogeneous of degree $k$, one finds
$$
\Re (\sigma_{T,x_0}(\mathrm{i}\zeta),\eta,\eta)= (\sigma_{T,x_0}(\mathrm{i}\zeta),\eta,\eta)\geq 
D_1|\zeta|^{2k} |\eta|^2,
$$
where
$$
D_1:=\min_{\zeta\rq{}\in\IR^m,\eta\rq{}\in\IC^\ell, |\zeta\rq{}|=1=
|\eta\rq{}|} |\sigma_{P,x_0}(\mathrm{i}\zeta\rq{})\eta\rq{}|^2>0.
$$
Furthermore, for $x\in U$ one easily gets
$$
\Re (\sigma_{A^{(t)},x}(\mathrm{i}\zeta),\eta,\eta)\geq 
-D(k,m)\max_{\alpha\in\IN^{m}_{2k}}|A^{(t)}_{\alpha}(x)||\zeta|^{2k} |\eta|^2,
$$
for some $D(k,m)>0$. From now one we fix some small $t$ such that 
$$
\sup_{x\in U}\max_{\alpha\in\IN^{m}_{2k}}|A^{(t)}_{\alpha}(x)|\leq D_1/(2D(k,m)).
$$
Then we get the estimate
$$
\Re (\sigma_{Q^{(t)}(\mathrm{i}\zeta),x},\eta,\eta)\geq\f{D_1}{2} |\zeta|^{2k} |\eta|^2
\text{  for all $x\in\IR^m$ },
$$
thus
$$
\left|\big(r^{2k}  +  \sigma_{Q^{(t)},x}(\mathrm{i}\xi)\big)^{-1}\right| \leq 
\min\{D_1/2,1\}(r + |\xi|)^{-2k}, 
$$
which is valid for all
$$
(x,\xi, r) \in \IR^m \times (\IR^m \times [0, \infty)) \setminus \{(0, 0)\}).
$$ 
In other words, $Q^{\psi}:=Q^{(t)}$ satisfies the assumptions of  Proposition \ref{pr2} with 
$\theta_0=\pi$, and by construction one has
\begin{align}\label{coin}
Q^{\psi}_{\alpha}=T_{\alpha}\text{ for all $\alpha\in\IN^{m}_{2k}$, in a open neighbourhood of 
$\mathrm{supp}(\psi)$}.
\end{align}
Since $ \mathsf{L}^1(\IR^m, \IC^\ell) \hookrightarrow \mathsf{B}^0_{1,\infty}(\IR^n, \IC^\ell)$, the 
assumption $f \in \mathsf{L}^1_{\mathrm{loc}} (U, \IC^\ell)$ implies
$$
\beta_0:= \left.\sup \big\{\beta\right|\>\beta\in\IR, \tilde{\psi} f \in 
\mathsf{B}^\beta_{1,\infty }(\IR^m, \IC^\ell)
\text{ for all $\tilde{\psi}\in\mathsf{C}^{\infty}_{\mathrm{c}}(U)$}\big\}\geq 0.
$$
We also know that $Pf \in \mathsf{L}^1_{\mathrm{loc}} (U, \IC^\ell)$. Then 
$P(\psi f) = \psi Pf + P_1 f$, where the commutator 
$P_1:=[P,\psi] \in \mathscr{D}_{\mathsf{C}^\infty}^{(k-1)}(U; \IC^\ell, \IC^\ell)$ has coefficients 
with compact support in $U$, and using (\ref{coin}) we get  
$$
Q^{\psi}(\psi f) = T(\psi f)= P^{\dagger} P(\psi f) = P^{\dagger} (\psi Pf) + P^{\dagger} P_1 f, 
$$
all equalities understood in the sense of distributions with compact support in $U$. We fix 
$R  \geq 0$ so large that the conclusions of Proposition \ref{pr2} hold for $Q=Q^{\psi}$, 
$\theta_0: = \pi$, $r = R$, 
$$
\beta \in \big\{-2k, \min\big\{\beta_0 +\f{1}{2}- 2k, -k\big\}\big\}.
$$
So $\psi f$ coincides with the unique solution 
$w$ in $\mathsf{B}^0_{1,\infty}(\IR^m, \IC^\ell)$ of 
\begin{equation}\label{eq2A}
R^{2k} w + Q^{\psi}w = R^{2k} \psi f + P^{\dagger} (\psi Pf) + P^{\dagger} P_1 f.
\end{equation}
On the other hand, as $\tilde{\psi} f \in \mathsf{B}^{\beta_0 - \frac{1}{2}}_{1,\infty}(\IR^m, \IC^\ell)$ 
for all $\tilde{\psi}\in\mathsf{C}^{\infty}_{\mathrm{c}}(U)$ (by the very definition of $\beta_0$), we get
$$
R^{2k} \psi f + P^{\dagger} (\psi Pf) + P^{\dagger} P_1 f \in 
\mathsf{B}^{\min\{-k,\beta_0+\frac{1}{2} -2k\}}_{1,\infty}(\IR^m, \IC^\ell).
$$
So (\ref{eq2A}) has a unique solution $\tilde{w}$ in 
$\mathsf{B}^{ \min\{\beta_0 + \frac{1}{2}, k\}}_{1,\infty}(\IR^m, \IC^\ell)$, 
evidently coinciding with $\psi f$, by the uniqueness of the solutions of (\ref{eq2A}) in the class 
$\mathsf{B}^0_{1,\infty}(\IR^m, \IC^\ell)$.  We deduce that 
$\psi f \in \mathsf{B}^{ \min\{\beta_0 + \frac{1}{2}, k\}}_{1,\infty}(\IR^m, \IC^\ell)$, so that, $\psi$ 
being arbitrary, $\min\big\{\beta_0 + \frac{1}{2}, k\big\} \leq \beta_0$, implying $k \leq \beta_0$ and 
$\min\{\beta_0 + \frac{1}{2}, k\} = k$. We have thus shown that 
$\psi f \in \mathsf{B}^k_{1,\infty}(\IR^m, \IC^\ell)$. 
\end{proof}

Keeping Remark \ref{adjo}.2 in mind, we immediately get the following characterization of local Sobolev spaces: 

\begin{Corollary} 
Let $E\to X$ be a smooth complex vector bundle, and let $k\in \IN_{\geq 0}$.\\
\emph{a)} If $p\in (1,\infty)$, then for any elliptic operator 
$Q\in \mathscr{D}_{\mathsf{C}^{\infty}}^{(k)}(X;E,E)$ one has 
\[
\Gamma_{\mathsf{W}^{k,p}_{\mathrm{loc}}}(X,E)=\Gamma_{\mathsf{W}^{Q,p}_{\mathrm{loc}}}(X,E).
\]
\emph{b)} For any elliptic $Q\in \mathscr{D}_{\mathsf{C}^{\infty}}^{(k+1)}(X;E,E)$ one has 
\[
\Gamma_{\mathsf{W}^{Q,1}_{\mathrm{loc}}}(X,E)\subset\Gamma_{\mathsf{W}^{k,1}_{\mathrm{loc}}}(X,E).
\] 
\end{Corollary}

Our second main result is the following abstract Meyers-Serrin type theorem:

\begin{Theorem}\label{main} 
Let $p\in [1,\infty)$, $s\in\IN$, $k_1\dots,k_s\in\IN_{\geq 0}$, and let $E\to X$, $F_i\to X$, for 
each $i\in\{1,\dots,s\}$, be smooth Hermitian vector bundles, and let $\mathfrak{P}:=\{P_1,\dots,P_s\}$ 
with $P_{i}\in\mathscr{D}_{\mathsf{C}^{\infty}}^{(k_i)}(X;E,F_i)$ be such that in case
 $k:=\max\{k_1,\dots,k_s\}\geq 2$ one has 
$\Gamma_{\mathsf{W}^{\mathfrak{P},p}_{\mu}}(X,E)\subset \Gamma_{\mathsf{W}^{k-1,p}_{\mathrm{loc}}}(X,E)$. 
Then for any $f\in \Gamma_{\mathsf{W}^{\mathfrak{P},p}_{\mu}}(X,E) $ 
there is a sequence 
$$
(f_n)\subset \Gamma_{\mathsf{C}^{\infty}}(X,E)\cap 
\Gamma_{\mathsf{W}^{\mathfrak{P},p}_{\mu}}(X,E),
$$
which can be chosen in $\Gamma_{\mathsf{C}^{\infty}_{\mathrm{c}}}(X,E)$ if $f$ is compactly 
supported, such that $\left\| f_n-f\right\|_{\mathfrak{P},p,\mu}\to 0$ as $n\to\infty$.
\end{Theorem}

The following vector-valued and higher order result on Friedrichs mollifiers is the main tool for 
the proof of Theorem \ref{main}, and should in fact be of an independent interest. 

\begin{Proposition}\label{molli} 
Let $0\leq h\in\mathsf{C}^{\infty}_{\mathrm{c}}(\IR^m)$  be such that $h(x)=0$ for all $x$ with 
$|x|\geq 1$, $\int_{\IR^m} h(x)\Id x =1$. For any $\epsilon>0$ define 
$0\leq h_{\epsilon}\in \mathsf{C}^{\infty}_{\mathrm{c}}(\IR^m)$ by 
$h_{\epsilon}(x):=\epsilon^{-m}h(\epsilon^{-1}x)$. Furthermore, let $U\subset \IR^m$ be open, let 
$k\in \IN_{\geq 0}$, $\ell_0,\ell_1\in\IN$, $p\in [1,\infty)$, and let 
$P\in\IDD^{(k)}_{\mathsf{C}^{\infty}}(U;\IC^{\ell_0},\IC^{\ell_1})$,
\[
P= \sum_{\alpha\in \IN^m_{k} }  P_{\alpha}\partial^{\alpha}, 
\>\text{ with $P_{\alpha}:U\longrightarrow\mathrm{Mat}(\IC;\ell_0\times \ell_1)$ in
$\mathsf{C}^{\infty}$.}
\]
\emph{a)} Assume that $f\in \mathsf{L}^{p}_{\mathrm{loc}}(U,\IC^{\ell_0})$, 
$Pf\in \mathsf{L}^{p}_{\mathrm{loc}}(U,\IC^{\ell_1})$, and that either 
$k<2$ or $f\in\mathsf{W}^{k-1,p}_{\mathrm{loc}}(U,\IC^{\ell_0})$. Then one 
has $Pf_{\epsilon}\to Pf$ as $\epsilon\to 0+$ in 
$\mathsf{L}^{p}_{\mathrm{loc}}(U,\IC^{\ell_1})$, where for sufficiently small 
$\epsilon>0$ we have set
\[
f_{\epsilon}:=\int_{\IR^m} h_{\epsilon}(\bullet-y)f(y)\Id y
\in\mathsf{C}^{\infty}(U,\IC^{\ell_0}).
\]
\emph{b)} If $f\in\mathsf{C}^k(U,\IC^{\ell_0})$, then $Pf_{\epsilon}\to P f$ as 
$\epsilon\to 0+$, uniformly over each $V\Subset  U$.
\end{Proposition}

\begin{proof} a) We prove the statement by an induction argument on the order of the operator similar 
to that in \cite[Appendix A]{Br}. The case $k=0$ is an elementary property of convolution, the 
case $k=1$ is the classical Friedrichs\rq{} theorem, see \cite{friedrichs}. Therefore, let $k\geq 2$ and 
assume that the result is true for operators of order at most $k-1$, and also that at least for 
some $\alpha\in\IN^m_k$ with $|\alpha|=k$ we have $P_\alpha\neq 0$. For 
$j\in\{1,\ldots,m\}$, let $e_j\in\IN^m_1$ be the $j$-th element of the canonical basis of $\IR^m$, set 
\[
J_j=\left.\big\{\alpha\right| \alpha\in\IN^m_k,\ |\alpha|=k,\ \alpha_j\geq 1\big\}, 
\]
and for $\alpha\in J_j$, set $\hat{\alpha}_j=\alpha-e_j$. For every 
$f\in\mathsf{W}^{k-1,p}_{\mathrm{loc}}(U,\IC^{\ell_0})$ such that 
$Pf\in \mathsf{L}^{p}_{\mathrm{loc}}(U,\IC^{\ell_1})$, $j\in\{1,\ldots,m\}$ with 
$J_j\neq\emptyset$ and $\alpha\in J_j$ we may write $g_j=\partial^{\hat{\alpha}_j}f$, and
$$
Pf = \sum_{j=1}^m \sum_{\alpha\in J_j} \partial_j(P_\alpha g_j) + Q f,
\> \text{ where $Q\in\IDD^{(k-1)}_{\mathsf{C}^{\infty}}(U;\IC^{\ell_0},\IC^{\ell_1})$.}
$$
By the induction hypothesis, $Qf_\epsilon \to Qf$ in $\mathsf{L}^{p}_{\mathrm{loc}}(U,\IC^{\ell_1})$ 
as $\epsilon\to 0+$. Moreover, by assumption $g_j\in \mathsf{L}^{p}_{\mathrm{loc}}(U,\IC^{\ell_0})$ 
for every $j$, hence $(g_j)_\epsilon \to g_j$ in $\mathsf{L}^{p}_{\mathrm{loc}}(U,\IC^{\ell_0})$ and as 
a consequence $(P_\alpha g_j)_\epsilon\to P_\alpha g_j$ in $\mathsf{L}^{p}_{\mathrm{loc}}(U,\IC^{\ell_1})$, 
so that $Pf_\epsilon \to Pf$ in $\mathsf{L}^{p}_{\mathrm{loc}}(U,\IC^{\ell_1})$ by Friedrichs' theorem, and 
the proof is complete.  \\
b) This follows from the following two well-known facts:  Firstly, if $f\in\mathsf{C}^{k}(U,\IC^{\ell_0})$, 
then $\partial^{\alpha}(f_{\epsilon})=(\partial^{\alpha}f)_{\epsilon}$ for all $\alpha\in\IN^m_k$ and 
all sufficiently small $\epsilon>0$. Secondly, if $g\in\mathsf{C}(U,\IC^{\ell_0})$, then for every 
$V\Subset U$ 
$$
\sup_{x\in V}|g_{\epsilon}(x)- g(x)|\to 0 \>\text{ as $\epsilon\to 0+$.}
$$
This completes the proof.
\end{proof}

\begin{proof}[Proof of Theorem \ref{main}.] Let 
$$
\ell_0:=\mathrm{rank}(E), \>\ell_j:=\mathrm{rank}(F_j),\>\> \text{ for any
$j\in\{1,\dots,s\}$.} 
$$
We take a relatively compact, locally finite atlas $\bigcup_{n\in\IN  }U_n=X$ such that 
each $U_n$ admits smooth orthonormal frames for 
\[
E\longrightarrow X, F_1\longrightarrow X,\dots,F_s\longrightarrow X.
\]
Let $(\varphi_n)$ be a partition of unity which is subordinate to $(U_n)$, that is, 
$$
0\leq \varphi_n\in\mathsf{C}^{\infty}_{\mathrm{c}}(U_n), \>\sum_n\varphi_n (x)=1
\>\text{ for all $x\in X$,}
$$
where the latter is a locally finite sum. Now let $f\in \Gamma_{\mathsf{W}^{\mathfrak{P},p}_{\mu}}(X,E)$, 
and $f_n:=\varphi_n f$. Let us first show that 
$f_n\in \Gamma_{\mathsf{W}^{\mathfrak{P},p}_{\mu,\mathrm{c}}}(U_n,E)$. 
Indeed, let $j\in\{1,\dots,s\}$. Then as elements of $\Gamma_{\mathsf{D}'}(U_n,E)$ one has 
$$
P_{j} f_n= \varphi_n P_{j}f+[P_j,\varphi_n] f,\>\text{ but }
\>[P_j,\varphi_n]\in \mathscr{D}^{k_j-1}_{\mathsf{C}^{\infty}}(U_n;E,F_j), 
$$
and as we have $f\in \Gamma_{\mathsf{W}^{k-1,p}_{\mathrm{loc}}}(X,E)$, it follows that
$$
\left(\partial^{\alpha}f_1,\dots,\partial^{\alpha}f_{\ell_0}\right)
\in\mathsf{L}^p_{\mathrm{loc}}(U_n,\IC^{\ell_0})
\>\text{ for all $\alpha\in\IN^m_{k-1}$,}
$$
where the $f_j$\rq{}s are the components of $f$ with respect to the smooth orthonormal frame on 
$U_n$ for $E$. Thus we get 
$$
[P_j,\varphi_n] f\in \Gamma_{\mathsf{W}^{\mathfrak{P},p}_{\mu,\mathrm{c}}}(U_n,E)
$$
as the coefficients of $[P_j,\varphi_n]$ have a compact support in $U_n$ and 
$0<\mu_{U_n}\in\mathsf{C}^{\infty}(U_n)$, and the proof of 
$f_n\in \Gamma_{\mathsf{W}^{\mathfrak{P},p}_{\mu,\mathrm{c}}}(U_n,E)$ is complete. 
But now, given $\epsilon >0$, we may appeal to Proposition \ref{molli} a) to pick an 
$f_{n,\epsilon}\in\Gamma_{\mathsf{C}^{\infty}_{\mathrm{c}}}(X,E)$ with support in $U_n$ such that 
$$
\left\|f_n- f_{n,\epsilon}\right\|_{\mathfrak{P},p,\mu}<\epsilon/2^{n+1}.
$$
Finally, $f_{\epsilon}(x):=\sum_n f_{n,\epsilon}(x)$, $x\in X$, is a locally finite sum and 
thus defines an element in $\Gamma_{\mathsf{C}^{\infty}}(X,E)$ which satisfies
$$
\left\|f_{\epsilon}-f\right\|_{\mathfrak{P},p,\mu}\leq 
\sum^{\infty}_{n=1}\left\|f_{n,\epsilon}-f_n\right\|_{\mathfrak{P},p,\mu}<\epsilon,
$$ 
which proves the first assertion. If $f$ is compactly supported, then picking a \emph{finite} covering 
of the support of $f$ with $U_n\rq{}s$ as above, the above proof also shows the second assertion.
\end{proof}

We close this section with the following example which shows that the assumptions of Theorem \ref{main} 
are optimal in a certain sense:

\begin{Example}\label{bep} 
Consider the third order differential operator 
$$
A:=-x\partial^3+(x-1)\partial^2=(1-\partial)\circ x\circ \partial^2\in\IDD^{(3)}_{\mathsf{C}^{\infty}}(\IR)
$$
on $\IR$ (with its Lebesgue measure). \emph{Then for any $p\in (1,\infty)$ one has 
$$
\mathsf{W}^{A,p}(\IR)\subset \mathsf{W}^{1,p}_{\mathrm{loc}}(\IR), \>\mathsf{W}^{A,p}(\IR)
\not\subset \mathsf{W}^{2,p}_{\mathrm{loc}}(\IR)
$$
and $\mathsf{W}^{A,p}(\IR)\cap \mathsf{C}^{\infty}(\IR)$ is not dense in $\mathsf{W}^{A,p}(\IR)$:} 
Indeed, we first observe that
$$
\mathsf{W}^{A,p}(\IR) = \{u | u\in\mathsf{L}^p(\IR) , x \partial^2 u \in \mathsf{W}^{1,p}(\IR)\}. 
$$
To see this, if $f = Au$ and $v = x \partial^2u$, $v \in \mathsf{S}'(\IR)$, $(1 - \partial) v = f$, 
so that $(1 - i\xi) \hat v = \hat f$, so that 
$v = \mathcal{F}^{-1} [(1 - i\xi)^{-1} \hat f]  \in \mathsf{W}^{1,p}(\IR)$. 
Here, $\mathcal{F}$ is the Fourier transformation and $\hat \Psi:=\mathcal{F}\Psi$.\\
Next we show $\mathsf{W}^{A,p}(\IR)\subset\mathsf{W}^{1,p}_{\mathrm{loc}}(\IR)$. In fact, let 
$u \in \mathsf{W}^{A,p}(\IR)$ and set $x \partial^2 u = g \in \mathsf{W}^{1,p}(\IR)$. We write $g$ 
in the form $g = g(0) + \int_0^x \partial g(y)\Id y$. Then
$$
\partial^2u(x) = \frac{g(0)}{x} + h(x), \quad x \in \IR \setminus \{0\}, 
$$
with $h(x) = \frac{1}{x} \int_0^x \partial g(y) \Id y$. As $p > 1$, it is a well known consequence of 
Hardy's inequality that $h \in \mathsf{L}^p(\IR)$. So
$$
\partial^2u = g(0) p. v. \left(\frac{1}{x}\right) + h + k, 
$$
with $k \in \mathsf{D}'(\IR)$, $\mathrm{supp}(k) \subseteq \{0\}$. We deduce that 
$$
g(x) = g(0) + x h(x) + x k(x),
$$
implying $x k(x) = 0$. From $k(x) = \sum_{j=0}^m a_j \delta^{(j)}$ it follows that  
$x k(x) = - \sum_{j=1}^m j a_j \delta^{(j-1)} = 0$ if and only if $k(x) = a_0 \delta$, whence 
$$
\partial^2u = g(0) p. v. \left(\frac{1}{x}\right) + h + a_0 \delta,
$$
so that
$$
\partial u(x) = g(0) \ln(|x|) + \int_0^x \partial g(y)\Id y + a_0 H(x) + C 
\in \mathsf{L}^p_{\mathrm{loc}}(\IR),
$$
where $H$ is teh Heaviside function, and we have proved that 
$\mathsf{W}^{A,p}(\IR) \subset \mathsf{W}^{1,p}_{\mathrm{loc}}(\IR)$. \\
In order to see $\mathsf{W}^{A,p}(\IR)\not\subset\mathsf{W}^{2,p}_{\mathrm{loc}}(\IR)$, 
consider the function $u(x) = \phi(x) \ln(|x|)$, with $\phi \in \mathsf{C}_c^{\infty}(\IR)$, $\phi(x) = x$ 
in some neighbourhood of $0$. Then $x \partial^2 u \in \mathsf{W}^{1,p}(\IR)$, but 
$u \not \in \mathsf{W}^{2,p}_{\mathrm{loc}}(\IR)$, since one has
\[
\partial^2 u(x) = p. v. \left(\frac{1}{x}\right)
\]
in a neighborhood of $0$. So Theorem 2.8 is not applicable. \\
To see that $\mathsf{W}^{A,p}(\IR) \cap \mathsf{C}^{\infty}(\IR)$ is not dense in $\mathsf{W}^{A,p}(\IR)$, 
let again $u(x):= \phi(x) \ln(|x|)$ with $\phi$ as above. Assume (by contradiction) that there exists 
$(u_n)_{n \in \IN}$ with $u_n \in \mathsf{W}^{A,p}(\IR) \cap \mathsf{C}^\infty(\IR)$, such that
$$
\|u_n - u\|_{\mathsf{L}^p(\IR)} + \|Au_n - Au\|_{\mathsf{L}^p(\IR)} \to 0 \text{ as $n \to \infty$}. 
$$
We set $v = x \partial^2 u$, $v_n = x \partial^2 u_n$. Then 
$$
\|v_n - v\|_{\mathsf{L}^p(\IR)}+\|\partial v_n - \partial  v\|_{\mathsf{L}^p(\IR)} \to 0
\text{ as $n \to \infty$},
$$
so that (considering the continuous representative of any $\mathsf{W}^{1,p}(\IR)$ equivalence class) 
$v_n (0) \to v(0)$. However, one has $v_n(0) = 0$ for all $n\in \IN$, while $v(0) = 1$, a contradiction. 
\end{Example}

\section{Applications of Theorem \ref{main}}

\subsection{The elliptic case}

Theorem \ref{regu} in combination with Remark \ref{adjo}.2 for formal adjoints immediately imply: 

\begin{Corollary}\label{ell} Let $s\in\IN$, $k_1\dots,k_s\in\IN_{\geq 0}$, let
$E\to X$, $F_i\to X$, $i\in\{1,\dots,s\}$, be smooth Hermitian vector bundles, and let 
$\mathfrak{P}:=\{P_1,\dots,P_s\}$ with $P_{i}\in\mathscr{D}_{\mathsf{C}^{\infty}}^{(k_i)}(X;E,F_i)$, 
and let $k:=\max\{k_1,\dots,k_s\}$. \vspace{1mm}

\emph{a)} Let $p\in (1,\infty)$. If one either has $k<2$, or the existence of some $j\in\{1,\dots,s\}$ 
with $P_j$ elliptic and $k_j\geq k-1$, then the assumptions from  Theorem \ref{main} are satisfied by 
$\mathfrak{P}$, in particular for any $f\in \Gamma_{\mathsf{W}^{\mathfrak{P},p}_{\mu}}(X,E) $ 
there is a sequence 
$$
(f_n)\subset \Gamma_{\mathsf{C}^{\infty}}(X,E)\cap 
\Gamma_{\mathsf{W}^{\mathfrak{P},p}_{\mu}}(X,E),
$$
which can be chosen in $\Gamma_{\mathsf{C}^{\infty}_{\mathrm{c} }}(X,E)$ if $f$ is compactly 
supported, such that $\left\| f_n-f\right\|_{\mathfrak{P},p,\mu}\to 0$ as $n\to\infty$. \\
\emph{b)} If one either has $k<2$, or the existence of some $j\in\{1,\dots,s\}$ with $P_j$ 
elliptic and $k_j=k$, then the assumptions from  Theorem \ref{main} are satisfied by 
$\mathfrak{P}$, in particular for any $f\in \Gamma_{\mathsf{W}^{\mathfrak{P},1}_{\mu}}(X,E) $ 
there is a sequence 
$$
(f_n)\subset \Gamma_{\mathsf{C}^{\infty}}(X,E)\cap 
\Gamma_{\mathsf{W}^{\mathfrak{P},1}_{\mu}}(X,E),
$$
which can be chosen in $\Gamma_{\mathsf{C}^{\infty}_{\mathrm{c}}}(X,E)$ if $f$ is compactly 
supported, such that $\left\| f_n-f\right\|_{\mathfrak{P},1,\mu}\to 0$ as $n\to\infty$.
\end{Corollary}

\subsection{A covariant Meyers-Serrin Theorem on arbitrary Rieman\-nian manifolds}\label{meyser}

The aim of this section is to apply Theorem \ref{main} in the context of covariant Sobolev spaces on 
Riemannian manifolds, which have been considered in this full generality, for example in \cite{salomonsen}, 
and in the scalar case, in \cite{aubin, hebey}. The point we want to make here is that Theorem \ref{main} 
can be applied in many situations, even if none of the underlying $P_j$\rq{}s is elliptic.\\
Let us start by recalling (cf. Section 3.3.1 in \cite{nico}) that if $E_j\to X$ is a smooth vector bundle 
and
$$
\nabla_j\in \mathscr{D}_{\mathsf{C}^{\infty}}^{(1)}\left(X;E_j,\mathrm{T}^* X \otimes E_j\right)
$$
a covariant derivative on $E_j\to X$ for $j=1,2$, then one defines the \emph{tensor
covariant derivative of $\nabla_1$ and $\nabla_2$} as the uniquely determined covariant derivative 
$$
\nabla_1\tilde{\otimes}\nabla_2\in  \mathscr{D}_{\mathsf{C}^{\infty}}^{(1)}
\left(X;E_1\otimes E_2,\mathrm{T}^* X \otimes E_1\otimes E_2\right)
$$
on $E_1\otimes E_2 \to X$ which satisfies 
\begin{align}
\nabla_1\tilde{\otimes}\nabla_2(f_1\otimes f_2)=\nabla_1(f_1)\otimes f_2+
f_1\otimes\nabla_2(f_2) \label{product}
\end{align}
for all $f_1\in\Gamma_{\mathsf{C}^{\infty}}(X,E_1)$,
$f_2\in\Gamma_{\mathsf{C}^{\infty}}(X,E_2)$ (the canonical isomorphism of 
$\mathsf{C}^{\infty}(X)$-modules 
$$
\Gamma_{\mathsf{C}^{\infty}}\left(X,\mathrm{T}^* X \otimes E_1\otimes E_2\right)
\longrightarrow \Gamma_{\mathsf{C}^{\infty}}\left(X,\mathrm{T}^* X \otimes E_2\otimes E_1\right)
$$
being understood).\\
Now let $(M,g)$ be a possibly noncompact smooth Riemannian manifold without boundary and let 
$\mu(\Id x)=\mathrm{vol}_g(\Id x)$ be the Riemannian volume measure. We also give ourselves 
a smooth Hermitian vector bundle $E\to M$ and let $\nabla$ be a Hermitian covariant derivative 
defined on the latter bundle. We denote the Levi-Civita connection on $\mathrm{T}^*M$ with
$\nabla_{g}$.  Then for any $j\in\IN$, the operator
$$
\nabla^{(j)}_g\in \mathscr{D}_{\mathsf{C}^{\infty}}^{(1)}
\big(M; \left( \mathrm{T}^*M\right)^{\otimes j-1} \otimes E, 
\left( \mathrm{T}^*M\right)^{\otimes j} \otimes E\big)
$$ 
is defined recursively by $\nabla^{(1)}_g:=\nabla$, 
$\nabla^{(j+1)}_g:=\nabla^{(j)}_g\tilde{\otimes} \nabla_{g}$, and we can further set
$$
\nabla^{j}_g:=\nabla^{(j)}_g\cdots \nabla^{(1)}_g\in 
\mathscr{D}_{\mathsf{C}^{\infty}}^{(j)}\big(M;E , \left(\mathrm{T}^*M\right)^{\otimes j} \otimes E\big)
$$
Note that if $\dim(M)>1$, then each $\nabla^{j}_g$ is nonelliptic. The following result makes 
Theorem \ref{main} accessible to covariant Riemannian Sobolev spaces:

\begin{Lemma}\label{formel2} 
Let $E\rq{}\to X$ be a smooth complex vector bundle with a covariant derivative $\nabla\rq{}$ 
defined on it. Then for any $p\in [1,\infty)$ one has 
$\Gamma_{\mathsf{W}^{\nabla\rq{},p}_{\mathrm{loc}}}(X,E\rq{})=
\Gamma_{\mathsf{W}^{1,p}_{\mathrm{loc}}}(X,E\rq{})$.
\end{Lemma}

\begin{proof} Let $\ell:=\mathrm{rank}(E\rq{})$, and pick  Hermitian structures on 
$E\rq{}$ and $\mathrm{T}^* X$. Given $f\in \mathsf{W}^{p,\nabla}_{\mathrm{loc}}(X,E\rq{})$, 
we have to prove $f\in\mathsf{W}^{1,p}_{\mathrm{loc}}(X,E\rq{})$. To this end, it is sufficient to 
prove that if $V\Subset W\Subset  X$ are such that there is a chart  
$$
x=(x^1,\dots,x^m):W\longrightarrow \IR^m 
$$
for $X$ in which $E\rq{}\to X$ admits a orthonormal frame 
$e_1,\dots,e_\ell\in\Gamma_{\mathsf{C}^{\infty}}(W,E\rq{})$, 
then with the components $f^j:=(f,e_j)$ of $f$ one has
\begin{align}
\sum_{k,j}\int_V |\partial_k f^j(x)|^p \Id x<\infty.\label{zuz}
\end{align}
To this end, note that there is a unique matrix of $1$-forms
$$
A\in \mathrm{Mat}\big(\Gamma_{\mathsf{C}^{\infty}}(W,\mathrm{T}^* X);\ell\times \ell\big)
$$
such that with respect to the frame $(e_j)$ one has $\nabla=\Id +A$, in the sense that for all 
$(\Psi^1,\dots,\Psi^\ell)\in\mathsf{C}^{\infty}(W,\IC^\ell)$ one has
\begin{align*}
\nabla \sum_{j}\Psi^j e_j =\sum_{j}(\Id \Psi^j )\otimes e_j+ \sum_{j}\sum_{i}\Psi^jA_{ij} \otimes e_i.
\end{align*}   
It follows that in $W$ one has
$$
\sum_{j}\Id f^j\otimes e_j=\Id f = \nabla f -A f,
$$
so using $|A_{ij}| \leq C$ in $V$ and that $(e_j)$ is orthonormal we arrive at
\begin{align}
\sum_j\int_V |\Id f^j(x)|_x^p \Id x\leq \tilde{C}\int_V |\nabla f(x)|^p_x\Id x<\infty.\label{diffy}
\end{align}
But it is well-known that the integrability (\ref{diffy}) implies (\ref{zuz}) (see for example 
Excercise 4.11 b) in \cite{buch}).
\end{proof}

With these preparations, we can state the following covariant Meyers-Serrin theorem for Riemannian 
manifolds (which in the case of scalar functions, that is, if $E=M\times\IC$ with $\nabla=\Id$) has 
also been observed in \cite[Lemma 3.1]{mueller}):

\begin{Corollary}\label{meyers}  Let $p\in
[1,\infty)$, $s\in\IN$, and define a global Sobolev space by
$$
\Gamma_{\mathsf{W}^{s,p}_{\nabla,g}}(M,E):=
\Gamma_{\mathsf{W}^{\{\nabla^1_g,...,\nabla^s_g\},p}_{\mathrm{loc}}}(M,E).
$$ 
Then one has
$$
\Gamma_{\mathsf{W}^{s,p}_{\nabla,g}}(M,E)\subset\Gamma_{\mathsf{W}^{s,p}_{\mathrm{loc}}}(M,E),
$$
in particular, for any $f\in\Gamma_{\mathsf{W}^{s,p}_{\nabla,g}}(M,E)$ there is a sequence
$$
(f_n)\subset \Gamma_{\mathsf{C}^{\infty}}(M,E)\cap \Gamma_{\mathsf{W}^{s,p}_{\nabla,g}}(M,E),
$$
which can be chosen in $\Gamma_{\mathsf{C}^{\infty}_{\mathrm{c}}}(M,E)$ if $f$ is compactly supported, such 
$$
\left\|f_n-f\right\|_{\nabla,g,p}:=
\left\|f_n-f\right\|_{\{\nabla^{1}_g,\dots,\nabla^{s}_g\},p,\mathrm{vol}_g}\to 0 \text{ as $n\to\infty$.}
$$
\end{Corollary}

\begin{proof} Applying Lemma \ref{formel2} inductively shows
$$
\Gamma_{\mathsf{W}^{s,p}_{\nabla,g}}(M,E)\subset\Gamma_{\mathsf{W}^{s,p}_{\mathrm{loc}}}(M,E),
$$
so that the other statements are implied by Theorem \ref{main}.
\end{proof}

\section{A substitute result for the $p=\infty$ case}

As ${\mathsf{C}^{\infty}}$ is not dense in ${\mathsf{L}^{\infty}}$,  it is clear that Theorem 
\ref{main} cannot be true for $p=\infty$. In this case, one can nevertheless smoothly approximate 
generalized $\mathsf{C}^k$-type spaces given by families $\mathfrak{P}$, without any further 
assumptions on $\mathfrak{P}$, an elementary fact which we record for the sake of completeness:

\begin{Proposition} Let $s\in\IN$, $k_1\dots,k_s\in\IN_{\geq 0}$, and let
$E\to X$, $F_i\to X$, for each $i\in\{1,\dots,s\}$, be 
smooth Hermitian vector bundles, and let $\mathfrak{P}:=\{P_1,\dots,P_s\}$ with 
$P_{i}\in\mathscr{D}_{\mathsf{C}^{\infty}}^{(k_i)}(X;E,F_i)$. Then with $k:=\max\{k_1,\dots,k_s\}$, 
define the Banach space $\Gamma_{\mathfrak{P},\infty}(X,E)$ by 
\begin{align*}
&\Gamma_{\mathfrak{P},\infty}(X,E)
\\
&:=\left.\Big\{f\right|f\in\Gamma_{\mathsf{C}\cap \mathsf{L}^{\infty}}(X,E), 
P_if\in\Gamma_{\mathsf{C}\cap \mathsf{L}^{\infty} }(X,F_i)\text{ \emph{for all} $i\in\{1,\dots,s\}$}\Big\}
\\
&\text{\emph{with norm} $\left\| f\right\|_{\mathfrak{P},\infty}:=
\left\|f\right\|_{\infty}+\sum^s_{i=1}\left\|P_{i} f\right\|_{\infty}$}.
\end{align*}
Assume that $\Gamma_{\mathfrak{P},\infty}(X,E) \subset \Gamma_{\mathsf{C}^{k-1}}(X,E)$. 
Then $\Gamma_{\mathsf{C}^{\infty}}(X,E)\cap\Gamma_{\mathfrak{P},\infty} (X,E)$ is dense  
in $\Gamma_{\mathfrak{P},\infty}(X,E)$.
\end{Proposition}

Using Proposition \ref{molli} b), this result follows from the same localization argument as in the proof of 
Theorem \ref{main}.

\appendix 
\section{An existence and uniqueness result for systems of linear elliptic PDE\rq{}s 
on the Besov scale}\label{beweis2}

Throughout this section, let $\ell\in\IN$ be arbitrary. We again use the notation $(\bullet,\bullet)$, 
$\left|\bullet\right|$, and $\mathrm{B}_r(x)$ for the standard Euclidean data in each $\IC^n$. 
We start by recalling the definition of Besov spaces with a positive differential order:

\begin{Definition} 
For any $\alpha\in (0,1], p\in [1,\infty], q\in [1,\infty)$, one defines 
$\mathsf{B}^{\alpha}_{p,q}(\IR^m,\IC^\ell)$ 
to be the space of $u\in\mathsf{L}^p(\IR^m,\IC^\ell)$ such that
\begin{align*}
\int_{\IR^m} \left\|u(\bullet+x)-2u+u(\bullet-x)\right\|_{\mathsf{L}^p(\IR^m,\IC^\ell)}
|x|^{-m-\alpha q}\Id x <\infty,
\end{align*}  
and $\mathsf{B}^{\alpha}_{p,\infty}(\IR^m,\IC^\ell)$ to be the space 
of $u\in\mathsf{L}^p(\IR^m,\IC^\ell)$ such that
\begin{align*}
\sup_{x\in\IR^m\setminus\{0\}}|x|^{-\alpha } 
\left\|u(\bullet+x)-2u+u(\bullet-x)\right\|_{\mathsf{L}^p(\IR^m,\IC^\ell)}
<\infty.
\end{align*}  
For $\alpha\in (1,\infty)$, $p\in [1,\infty]$, $q\in [1,\infty]$, one defines 
$\mathsf{B}^{\alpha}_{p,q}(\IR^m,\IC^\ell)$ to be the space\footnote{Here, 
$[\alpha]:=\max\{j|j\in\IN, j<\alpha\}$ } of $u\in\mathsf{W}^{[\alpha],p}(\IR^m,\IC^\ell)$ 
such that for all $\beta\in (\IN_{\geq 0})^m$ with $|\beta|=[\alpha]$ 
one has $\partial^{\beta}u \in\mathsf{B}^{\alpha-[\alpha],p}(\IR^m,\IC^\ell)$. These are 
Banach spaces with respect to their canonical norms.
\end{Definition}

For negative differential orders, the definition is more subtle:

\begin{Propandef} Let $t(\zeta):=|\zeta|$, $\zeta\in\IR^m$, and for any $\gamma\in\IR$ let 
$$
J_{\gamma}:=\mathcal{F}^{-1}(1+t^2)^{-\gamma/2}
$$ 
denote the Bessel potential of order $\gamma$. Let $\alpha\in (-\infty,0]$, $p\in [1,\infty]$, 
$q\in [1,\infty)$, and pick some $\beta\in (0,\infty)$. Then one defines 
$\mathsf{B}^{\alpha}_{p,q}(\IR^m,\IC^\ell)$ to be the space of $u\in\mathsf{S}\rq{}(\IR^m,\IC^\ell)$ 
such that $u= J_{\alpha-\beta}*f$ for some $f\in \mathsf{B}^{\beta}_{p,q}(\IR^m,\IC^\ell)$. 
This definition does not depend on the particular choice of $\beta$, and one defines 
$$
\left\|u\right\|_{\mathsf{B}^{\alpha}_{p,q}(\IR^m,\IC^\ell)}:=
\left\|J_{\alpha-1}*u\right\|_{\mathsf{B}^{1}_{p,q}(\IR^m,\IC^\ell)},
$$
which again produces a Banach space.
\end{Propandef}

We are going to prove:

\begin{Proposition}\label{pr2}
Let $n\in\IN_{\geq 0}$, $Q \in \mathscr{D}_{\mathsf{C}^\infty}^{(n)}(\IR^m; \IC^\ell, \IC^\ell)$,
$$
\begin{array}{ccccc}
Q = \sum_{\alpha \in \IN_n^m}Q_\alpha \partial^\alpha, & {\it with }& Q_\alpha : \IR^m 
\longrightarrow 
\mathrm{Mat}(\IC; \ell \times \ell) & {\it in } &  \mathsf{W}^{\infty,\infty},
\end{array}
$$
that is, $Q_\alpha$ and all its derivatives are bounded. Suppose also that for some 
$\theta_0 \in (-\pi, \pi]$ and all
$$
 (x,\xi, r) \in \IR^m \times (\IR^m \times [0, \infty)) \setminus \{(0, 0)\}),
$$ 
the complex $\ell\times \ell$ matrix $r^n \mathrm{e}^{\mathrm{i}\theta_0}-\sigma_{Q,x}(\mathrm{i}\xi)$ 
is invertible, and that there are is $C>0$ such that for all $(x,\xi, r)$ as above one has
\begin{align}\label{ellp}
\left|\big(r^n \mathrm{e}^{\mathrm{i}\theta_0}   -  \sigma_{Q,x}(\mathrm{i}\xi)\big)^{-1}\right| 
\leq C(r + |\xi|)^{-n}. 
\end{align}
We consider the system of linear PDE\rq{}s given by 
\begin{equation}\label{eq1}
r^n \mathrm{e}^{\mathrm{i}\theta_0} u(x) - Q u(x) = g(x), \quad x \in \IR^m, r\geq 0. 
\end{equation}
Then for any $\beta \in \IR$, $ p, q\in [1,\infty]$, there is a $R =R(\beta,p,q,Q)\geq 0$ with the 
following property: if $r \geq R$ and $g \in \mathsf{B}^\beta_{p,q}(\IR^m, \IC^\ell)$, then 
(\ref{eq1}) has a unique solution $u\in\mathsf{B}^{\beta+n}_{p,q}(\IR^m, \IC^\ell)$.
\end{Proposition}

Note that given some $Q \in \mathscr{D}_{\mathsf{C}^\infty}^{(n)}(\IR^m; \IC^\ell, \IC^\ell)$ which is 
strongly elliptic in the usual sense 
$$
\Re(\sigma_{Q,x}(\zeta)\eta,\eta)\geq \tilde{C} |\eta|^2\> \text{ for all $x\in \IR^m$, 
$\eta\in\IC^\ell$, $\zeta\in \IC^m$ with $|\zeta|=1$}
$$
with some $\tilde{C}>0$ which is uniform in $x$, $\eta$, $\zeta$, it is straightforward to see that 
the condition (\ref{ellp}) is satisfied with $\theta_0=\pi$, $C=\min\{1,\tilde{C}\}$ (see also the 
proof of Theorem \ref{regu} b)).\\
Before we come to the proof of Proposition \ref{pr2}, we first collect some well known facts concerning 
Besov spaces. Unless otherwise stated, the reader may find these results in \cite{davide} and the 
references therein. 

(i) For every $p \in [1, \infty]$ one has 
$\mathsf{B}_{p,1}^0(\IR^m) \hookrightarrow \mathsf{L}^p(\IR^m) \hookrightarrow \mathsf{B}_{p,\infty}^0(\IR^m)$. 

(ii) Let $ p, q \in  [1, \infty]$, $ \beta \in \IR$. Then 
$$
\mathsf{B}^{\beta +1}_{p,q} (\IR^m) = \{f| f \in \mathsf{B}^{\beta}_{p,q} (\IR^m),  
\partial_j f \in \mathsf{B}^{\beta}_{p,q} (\IR^m) \text{ for all } j \in \{1, \dots, m\}\}. 
$$
So for all $ k \in \IN$ one has 
$\mathsf{B}_{p,1}^k(\IR^m) \hookrightarrow \mathsf{W}^{k,p}(\IR^m) 
\hookrightarrow \mathsf{B}_{p,\infty}^k(\IR^m)$. 

(iii) As a consequence of (ii), we have the following particular case of Sobolev embedding 
theorem: if $\beta \in \IR$, $1 \leq p, q \leq \infty$, 
$\mathsf{B}^\beta_{p,q}(\IR^m) \hookrightarrow \mathsf{B}^{\beta - m/p}_{\infty,\infty}(\IR^m)$. 

(iv) Let us indicate with $(\cdot,\cdot)_{\theta,q}$ ($0 < \theta < 1$, $1 \leq q \leq \infty$) 
the real interpolation functor. Then, if 
$-\infty < \alpha_0 < \alpha_1 < \infty$, $1 \leq p, q_0, q_1 \leq \infty$, the real interpolation 
space $(\mathsf{B}^{\alpha_0}_{p,q_0}(\IR^m), \mathsf{B}^{\alpha_1}_{p,q_1}(\IR^m))_{\theta,q}$ coincides 
with $\mathsf{B}^{(1-\theta)\alpha_0 + \theta \alpha_1}_{p,q}(\IR^m)$, with equivalent norms. 

(v) If $1 \leq p, q < \infty$ and $\beta \in \IR$, the  antidual space of 
$\mathsf{B}^\beta_{p,q}(\IR^m)$ can be identified with $\mathsf{B}^{-\beta}_{p',q'}(\IR^m)$ 
in the following sense: if $g \in \mathsf{B}^{-\beta}_{p',q'}(\IR^m)$, then the (antilinear) 
distribution $\left\langle\bullet, \overline g\right\rangle$ can be uniquely extended to a 
bounded antilinear functional in $\mathsf{B}^\beta_{p,q}(\IR^m)$ (we recall  
here also that, whenever $\max\{p,q\}< \infty$, then $\mathsf{C}_{\mathrm{c}}^\infty(\IR^m)$ 
is dense in each $\mathsf{B}^\beta_{p,q}(\IR^m)$). Moreover, all bounded antilinear functionals on 
$\mathsf{B}^\beta_{p,q}(\IR^m)$ can be obtained in this way.

(vi) Suppose that $a \in \mathsf{C}^\infty(\IR^m)$, and that for some $n \in \IR$ and all $\xi\in\IR^m$ 
one has 
$$
\max_{\alpha\in\IN^{m}_{m+1} } |\partial^\alpha a(\xi)|\leq C (1 + |\xi|)^{n - |\alpha|}. 
$$
Then for all 
$$ 
(\beta,p,q) \in \IR \times [1, \infty] \times [1, \infty],
$$ 
the Fourier multiplication operator $f \mapsto {\mathcal F}^{-1}(a {\mathcal F} f)$ maps 
$\mathsf{B}^\beta_{p,q}(\IR^m)$ into $\mathsf{B}^{\beta-n}_{p,q}(\IR^m)$, and the 
norm of the latter operator can be estimated by 
$$
C \sup_{\alpha \in \IN^m_{m+1},\xi\in\IR^m} \left|(1 + |\xi|)^{|\alpha|-n} \partial^\alpha a(\xi)\right|,
$$ 
for some $C >0$ independent of $a$ (cf. \cite{amann}). 
 
(vii) If $a \in \mathsf{W}^{\infty,\infty}(\IR^m)$ and $f \in \mathsf{B}^\beta_{p,q}(\IR^m)$, then one has 
$af \in \mathsf{B}^\beta_{p,q}(\IR^m)$. More precisely, there exist $C>0$, $N \in \IN$, independent of $a$ 
and $f$, such that 
$$
\|af\|_{\mathsf{B}^\beta_{p,q}(\IR^m)} \leq C\left(\|a\|_{\mathsf{L}^\infty(\IR^m)}
\|f\|_{\mathsf{B}^\beta_{p,q}(\IR^m)} + \|a\|_{\mathsf{W}^{N,\infty}(\IR^m)}
\|f\|_{\mathsf{B}^{\beta-1}_{p,q}(\IR^m)}\right). 
$$

(viii) Let $0\leq \chi_0 \in C_{\mathrm{c}}^\infty(\IR^m)$ be such that for some $\delta >0$ one 
has
$$
\mathrm{supp}(\chi_0)\subset[-\delta, \delta]^m, \>\chi_0 = 1 \text{ in } [-\delta/2, \delta/2]^m. 
$$
For any $j \in \IZ^m$ set
$$
\chi_j(x):= \chi_0(x -  \delta j/2  ), \chi(x):= 
\sum_{j \in \IZ^m} \chi_j(x), \psi_j(x) := \frac{\chi_j(x)}{\chi(x)}. 
$$
Then for all $ \beta \in \IR$, $ p \in [1, \infty]$, there exist $C_1,C_2>0$ such that for all 
$f \in \mathsf{B}^\beta_{p,p}(\IR^m)$ it holds that
$$
C_1 \|f\|_{\mathsf{B}^\beta_{p,p}(\IR^m)} \leq \| (\|\psi_j f\|_{\mathsf{B}^\beta_{p,p}
(\IR^m)})_{j \in \IZ^m}\|_{\ell ^p(\IZ^m)} \leq C_2 \|f\|_{\mathsf{B}^\beta_{p,p}(\IR^m)}. 
$$
With these preparations, we can now give the proof of Proposition \ref{pr2}:

\begin{proof}[Proof of Proposition \ref{pr2}] We prove the result in several steps. 

\medskip
 
{\it Step 1 (constant coefficients):  Let 
$$
Q = \sum_{\alpha \in \IN_n^m}Q_\alpha \partial^\alpha, 
\text{ with }Q_\alpha \in \mathrm{Mat}(\IC; \ell \times \ell), 
$$
and suppose that for some $\theta_0 \in (-\pi, \pi]$ and all
$$
 (\xi, r) \in (\IR^m \times [0, \infty)) \setminus \{(0, 0)\}),
$$
the $l \times l$ matrix $r^n \mathrm{e}^{\mathrm{i}\theta_0}  - \mathrm{i}^n \sigma_{Q}(\xi)$ is 
invertible, and that there exists $C >0$ such that for all $(\xi, r)$ as above one has
\begin{equation}\label{eq4}
|(r^n \mathrm{e}^{\mathrm{i}\theta_0}  -  \sigma_{Q}(\mathrm{i}\xi))^{-1}| \leq C(r + |\xi|)^{-n}. 
\end{equation}
Then for any $\beta \in \IR$, $1 \leq p, q \leq \infty$, there exists $R \geq 0$ such that, if $r \geq R$ 
and $g \in \mathsf{B}^\beta_{p,q}(\IR^m, \IC^\ell)$, the system (\ref{eq1}) has a unique solution 
$u\in\mathsf{B}^{\beta+n}_{p,q}(\IR^m, \IC^\ell)$. Moreover,  there exists a constant $C_0 >0$, 
which only depends on $\beta$, $p, q$, the constant $C$ in (\ref{eq4}) and on 
${\displaystyle \max_{\alpha \in \IN^m_n}}$ $|Q_\alpha|$, such that for all 
$r \geq R$ one has
$$
r^n \|u\|_{\mathsf{B}^\beta_{p,q}(\IR^m, \IC^\ell)} + \|u\|_{\mathsf{B}^{\beta+n}_{p,q}(\IR^m, \IC^\ell)} 
\leq  C_0 \|g\|_{\mathsf{B}^\beta_{p,q}(\IR^m, \IC^\ell)}. 
$$
By interpolation, we obtain also, for every $\theta \in [0, 1]$ and $r \geq R$, 
\begin{equation}\label{eq5A}
\|u\|_{\mathsf{B}^{\beta+\theta n}_{p,q}(\IR^m, \IC^\ell)} \leq 
C_0 r^{(\theta - 1)n}\|g\|_{\mathsf{B}^\beta_{p,q}(\IR^m, \IC^\ell)}. 
\end{equation}
}

\medskip

In order to prove the statement from Step 1, we start by assuming that $Q$ coincides with its principal 
part $Q_n: = \sum_{|\alpha| = n} Q_\alpha \partial^\alpha$. Then, employing the Fourier transform, 
it is easily seen that for any $r\geq 0$, $g\in \mathsf{S}'(\IR^m, \IC^\ell)$, the only possible solution 
$u\in\mathsf{S}'(\IR^m, \IC^\ell)$ of (\ref{eq1}) is 
$$
u = {\mathcal F}^{-1}\left((r^n \mathrm{e}^{\mathrm{i}\theta_0} -  
\sigma_{Q}(\mathrm{i}\xi))^{-1} {\mathcal F} g\right). 
$$
Observe that $(r^n \mathrm{e}^{\mathrm{i}\theta_0} -  \sigma_{Q}(\mathrm{i}\xi))^{-1}$ is 
positively homogeneous of degree $-n$ in the variables 
$$
(r, \xi) \in ([0, \infty) \times \IR^m) \setminus \{(0,0)\}.
$$
So for all $ \alpha \in \IN_n^m$, the matrix  
$\partial_\xi^\alpha (r^n \mathrm{e}^{\mathrm{i}\theta_0}-  \sigma_{Q}(\mathrm{i}\xi))^{-1}$ 
is positively homogeneous of degree $-n - |\alpha|$ in these variables, implying
$$
\left|\partial_\xi^\alpha(r^n \mathrm{e}^{\mathrm{i}\theta_0} -  \sigma_{Q}(\mathrm{i}\xi))^{-1}\right| 
\leq C(\alpha) (r + |\xi|)^{-n-|\alpha|}. 
$$
It is easily seen that $C(\alpha)$ can be estimated in terms of the constant $C$ in (\ref{eq4}) and 
of ${\displaystyle \max_{\alpha \in \IN^m_n}}$ $|Q_\alpha|$. We deduce from (vi) that, for 
all $r \geq  0$, and all  $g \in \mathsf{B}^{\beta}_{p,q}(\IR^m, \IC^\ell)$, the problem
\begin{equation}\label{eq5}
r^n \mathrm{e}^{\mathrm{i}\theta_0}u(x) - Q_n u(x) = g(x), \quad x \in \IR^m
\end{equation}
has a unique solution $u$ in $\mathsf{B}^{\beta+n}_{p,q}(\IR^m, \IC^\ell)$, and also that for all 
$r_0>0$ there is $C(r_0)>0$ such that for all $r \geq r_0 $ one has 
$$
\|u\|_{\mathsf{B}^{\beta+n}_{p,q}(\IR^m, \IC^\ell)} \leq 
C(r_0) \|g\|_{\mathsf{B}^{\beta}_{p,q}(\IR^m, \IC^\ell)}. 
$$
The latter inequality together with (\ref{eq5}) also gives 
$$
\begin{array}{ll}
\|u\|_{\mathsf{B}^{\beta}_{p,q}(\IR^m, \IC^\ell)} & \leq 
r^{-n} (\|g\|_{\mathsf{B}^{\beta}_{p,q}(\IR^m, \IC^\ell)} + 
\|Q_n u\|_{\mathsf{B}^{\beta}_{p,q}(\IR^m, \IC^\ell)}) 
\\ \\
&  \leq C_1(r_0) r^{-n} (\|g\|_{\mathsf{B}^{\beta}_{p,q}(\IR^m, \IC^\ell)} + 
\|u\|_{\mathsf{B}^{n+\beta}_{p,q}(\IR^m, \IC^\ell)}) 
\\ \\
&\leq C_2(r_0) r^{-n}\|g\|_{\mathsf{B}^{\beta}_{p,q}(\IR^m, \IC^\ell)},
\end{array}
$$
and now the estimate (\ref{eq5A}) follows directly by interpolation (see (iv)). Now we extend the 
previous facts from $Q_n$ to $Q$, taking $r$ sufficiently large. In fact, we write (\ref{eq1}) in the form
$$
r^n \mathrm{e}^{\mathrm{i}\theta_0}u(x) - Q_n u(x) = (Q - Q_n)u(x) + g(x).
$$
Taking $h:= r^n \mathrm{e}^{\mathrm{i}\theta_0}u - Q_n u$ as new unknown, we obtain 
\begin{equation}\label{eq7}
h - (Q - Q_n) (r^n \mathrm{e}^{\mathrm{i}\theta_0} - Q_n)^{-1} h = g.
\end{equation}
We have
$$
\begin{array}{ll}
&\|(Q - Q_n) (r^n \mathrm{e}^{\mathrm{i}\theta_0} - Q_n)^{-1}h\|_{\mathsf{B}^{\beta}_{p,q}(\IR^m, \IC^\ell} )
\\ \\
&\leq C_0 \|(r^n \mathrm{e}^{\mathrm{i}\theta_0} - Q_n)^{-1}h\|_{\mathsf{B}^{\beta+n-1}_{p,q}(\IR^m, \IC^\ell)}  
\leq C_1 r^{-1} \|h\|_{\mathsf{B}^{\beta}_{p,q}(\IR^m, \IC^\ell)}.
\end{array}
$$
So, if $C_1 r^{-1} < 1$, then (\ref{eq7}) has a unique solution 
$h\in\mathsf{B}^{\beta}_{p,q}(\IR^m, \IC^\ell)$ 
and, in case $C_1 r^{-1} \leq \frac{1}{2}$ such solution can be estimated in the form
$$
\|h\|_{\mathsf{B}^{\beta}_{p,q}(\IR^m, \IC^\ell)} \leq 2 \|g\|_{\mathsf{B}^{\beta}_{p,q}(\IR^m, \IC^\ell)}.
$$
So the previous estimates and results can be extended from $Q_n$ to $Q$. 

\medskip

{\it Step 2 (a priori estimate for solutions in $\mathsf{B}^{\beta+n}_{p,q}$ with small support): Let 
$\beta \in \IR$, $1 \leq p, q \leq \infty$. Then there exist $r_0 ,\delta, C >0$ with the following property: 
if $u \in \mathsf{B}^{\beta+n}_{p,q}(\IR^m, \IC^\ell)$ satisfies
$$
r^n \mathrm{e}^{\mathrm{i}\theta_0} u - Q u = g,\> \mathrm{supp}(u)\subset 
\prod_{j=1}^m [x^0_j - \delta, x^0_j + \delta]\text{ for some $x^0 \in \IR^m$},
$$ 
then one has 
\begin{equation}\label{eq8}
r^n \|u\|_{\mathsf{B}^\beta_{p,q}(\IR^m, \IC^\ell)} + \|u\|_{\mathsf{B}^{\beta+n}_{p,q}(\IR^m, \IC^\ell)} 
\leq C \|g\|_{\mathsf{B}^\beta_{p,q}(\IR^m, \IC^\ell)}. 
\end{equation}
}

\medskip

\noindent
In order to prove this, we define the constant coefficient operator 
$Q(x_0,\partial):=\sum_{\alpha\in\IN^m_n}Q_{\alpha}(x_0)\partial^{\alpha}$ and observe that 
$$
r^n \mathrm{e}^{\mathrm{i}\theta_0} u(x) - Q(x^0,\partial)u(x) = (Q - Q(x^0,\partial))u(x) + g(x). 
$$
Let $\epsilon >0$. For any $\phi \in \mathsf{C}_{\mathrm{c}}^\infty (\IR^m)$ which satisfies
\begin{align*}
&\mathrm{supp}(\phi)\subset \prod_{j=1}^m [x^0_j - 2\delta, x^0_j + 2\delta], 
\\
&\phi = 1 \text{ in } \prod_{j=1}^m [x^0_j - \delta, x^0_j + \delta], \>\|\phi\|_{\mathsf{L}^\infty(\IR^m)} 
= 1,
\end{align*}
we have
$$
(Q - Q(x^0,\partial))u = \phi (Q - Q(x^0,\partial))u.
$$ 
So, taking $\delta$ sufficiently small, from (iv) and (vii) we obtain 
$$
\|(Q - Q(x^0,\partial))u\|_{\mathsf{B}^{\beta}_{p,q}(\IR^m, \IC^\ell)} \leq 
\epsilon \|u\|_{\mathsf{B}^{\beta+n}_{p,q}(\IR^m, \IC^\ell)} + 
C(\epsilon) \|u\|_{\mathsf{B}^{\beta+n-1}_{p,q}(\IR^m, \IC^\ell)}.
$$
Observe that $\delta$ can be chosen independent of $x^0$. So, from Step 1 with $\theta=(n-1)/n$ in 
\eqref{eq5A}, taking $r$ sufficiently large (uniformly in $x^0$) we obtain
\begin{equation*}
\begin{array}{c}
r^n \|u\|_{\mathsf{B}^{\beta}_{p,q}(\IR^m, \IC^\ell)} + r\|u\|_{\mathsf{B}^{\beta+n-1}_{p,q}(\IR^m, \IC^\ell)} + 
\|u\|_{\mathsf{B}^{\beta+n}_{p,q}(\IR^m, \IC^\ell)}  
\\ \\
\leq C_0\left(\epsilon \|u\|_{\mathsf{B}^{\beta+n}_{p,q}(\IR^m, \IC^\ell)} + 
C(\epsilon) \|u\|_{\mathsf{B}^{\beta+n-1}_{p,q}(\IR^m, \IC^\ell)} + 
\|g\|_{\mathsf{B}^{\beta}_{p,q}(\IR^m, \IC^\ell)}\right). 
\end{array}
\end{equation*}
Taking $\epsilon$ so small that $C_0 \epsilon \leq \frac{1}{2}$ and $r$ so large that 
$C_0 C(\epsilon) \leq r$, we deduce (\ref{eq8}). 

\medskip

{\it Step 3 (a priori estimate for arbitrary solutions in $\mathsf{B}^{\beta+n}_{p,p}$): For any 
$\beta \in \IR$, $p \in [1, \infty)$, there exist $C_0, r_0>0$ such that if $r \geq r_0$ and 
$u \in \mathsf{B}^{\beta +n}_{p,p}(\IR^m; \IC^\ell)$ is a solution to (\ref{eq1}), then
\begin{equation}
r^n \|u\|_{\mathsf{B}^\beta_{p,p}(\IR^m, \IC^\ell)} + \|u\|_{\mathsf{B}^{\beta+n}_{p,p}(\IR^m, \IC^\ell)} 
\leq  C_0 \|g\|_{\mathsf{B}^\beta_{p,p}(\IR^m, \IC^\ell)}. 
\end{equation}
}

\medskip

To see this, we take $\delta,r_0>0$ so that the conclusion in Step 2 holds. We consider a family of 
functions $(\psi_j)_{j \in \IZ^m}$ as in (viii). Let $u \in \mathsf{B}^{\beta+n}_{p,p}(\IR^m, \IC^\ell)$ 
solve (\ref{eq1}), with $r \geq r_0$. For each $j \in \IZ^m$ we have 
$$
r^n \psi_j u - Q(\psi_j u) = \psi_j g + Q_j u,
$$
with the commutator
$$
Q_j:=[Q,\psi_j] = \sum_{1 \leq |\alpha| \leq n} Q_\alpha \sum_{\gamma < \alpha} 
{\alpha \choose \gamma} \partial^{\alpha - \gamma} \psi_j \partial^\gamma.
$$
We set
$$
\IZ_j:= \{i| \>i \in \IZ^m, \>  \mathrm{supp}(\psi_i) \cap \mathrm{supp}(\psi_j) \neq \emptyset\}. 
$$
Then $Q_j u = \sum_{i \in \IZ_j} Q_j(\psi_i u)$, so that
$$
\|Q_j u\|_{\mathsf{B}^\beta_{p,p}(\IR^m, \IC^\ell)} \leq C_1 \sum_{i \in \IZ_j} 
\|\psi_i u\|_{\mathsf{B}^{\beta+n-1}_{p,p}(\IR^m, \IC^\ell)}. 
$$
with $C_1$ independent of $j$. So, from Step 2, we have, for each $j \in \IZ^m$, 
\begin{equation}\label{eq12}
\begin{array}{c}
r^n \|\psi_j u\|_{\mathsf{B}^\beta_{p,p}(\IR^m, \IC^\ell)} + 
r \|\psi_j u\|_{\mathsf{B}^{\beta+n-1}_{p,p}(\IR^m, \IC^\ell)} +  
\|\psi_j u\|_{\mathsf{B}^{\beta+n}_{p,p}(\IR^m, \IC^\ell)} 
\\ \\
\leq C_2\left( \|\psi_j g\|_{\mathsf{B}^\beta_{p,p}(\IR^m, \IC^\ell)} + 
\sum_{i \in \IZ_j} \|\psi_i u\|_{\mathsf{B}^{\beta+n-1}_{p,p}(\IR^m, \IC^\ell)}\right).
\end{array}
\end{equation}
We observe that $\IZ_j$ has at most $7^m$ elements. So we have, in case $p < \infty$, 
$$
\Big(\sum_{i \in \IZ_j} \|\psi_i u\|_{\mathsf{B}^{\beta+n-1}_{p,p}(\IR^m, \IC^\ell)}\Big)^p \leq 
7^{m(p-1)} \sum_{i \in \IZ_j} \|\psi_i u\|_{\mathsf{B}^{\beta+n-1}_{p,p}(\IR^m, \IC^\ell)}^p
$$
and
\begin{align*}
&\sum_{j \in \IZ^m}\Big(\sum_{i \in \IZ_j} \|\psi_i u\|_{\mathsf{B}^{\beta+n-1}_{p,p}
(\IR^m, \IC^\ell)}\Big)^p 
\\
&\leq 7^{m(p-1)} \sum_{j \in \IZ^m}\sum_{i \in \IZ_j} \|\psi_i u\|_{\mathsf{B}^{\beta+n-1}_{p,p}
(\IR^m, \IC^\ell)}^p 
\\  
&= 7^{m(p-1)}  \sum_{i \in \IZ^m} \Big(\sum_{j \in \IZ_i} 1\Big) 
\|\psi_i u\|_{\mathsf{B}^{\beta+n-1}_{p,p}(\IR^m, \IC^\ell)}^p 
\\  
&\leq 7^{mp} 
\left\|(\|\psi_i u\|_{\mathsf{B}^{\beta+n-1}_{p,p}(\IR^m, \IC^\ell)})_{i \in \IZ^m}\right\|_{\ell^p(\IZ^m)}^p. 
\end{align*}
So, from (\ref{eq12}) and (viii), we deduce
 \begin{align*}
&r^n \|u\|_{\mathsf{B}^\beta_{p,p}(\IR^m, \IC^\ell)} + 
r \|u\|_{\mathsf{B}^{\beta+n-1}_{p,p}(\IR^m, \IC^\ell)} +  
\|u\|_{\mathsf{B}^{\beta+n}_{p,p}(\IR^m, \IC^\ell)} 
\\  
&\leq C_3 \left(r^n \left\|\big(\|\psi_j u\|_{\mathsf{B}^\beta_{p,p}
(\IR^m, \IC^\ell)}\big)_{j \in \IZ^m} \right\|_{\ell^p(\IZ^m)}\right. 
\\ 
&\>\>\>\>\>\>\>\>\>\>\>\> + 
r \left\|\big(\|\psi_j u\|_{\mathsf{B}^{\beta+n-1}_{p,p}(\IR^m, \IC^\ell)}
\big)_{j \in \IZ^m} \right\|_{\ell^p(\IZ^m)} 
\\
&\>\>\>\>\>\>\>\>\>\>\>\> 
\left.+  \left\|\big(\|\psi_j u\|_{\mathsf{B}^{\beta+n}_{p,p}
(\IR^m, \IC^\ell)}\big)_{j \in \IZ^m} \right\|_{\ell^p(\IZ^m)}\right)
\\  
&\leq C_4 \left(\left\|\big( \|\psi_j g\|_{\mathsf{B}^\beta_{p,p}
(\IR^m, \IC^\ell)}\big)_{j \in \IZ^m} \right\|_{\ell^p(\IZ^m)} \right.
\\
&\>\>\>\>\>\>\>\>\>\>\>\>+ \left. \left\|\big(\|\psi_j u\|_{\mathsf{B}^{\beta+n-1}_{p,p}
(\IR^m, \IC^\ell)}\big)_{j \in \IZ^m} \right\|_{\ell^p(\IZ^m)}\right) 
\\  
&\leq C_5 \left(\|g\|_{\mathsf{B}^\beta_{p,p}(\IR^m, \IC^\ell)} + \|u\|_{\mathsf{B}^{\beta+n-1}_{p,p}
(\IR^m, \IC^\ell)}\right). 
\end{align*}
Taking $r \geq C_5$, we get the conclusion. 

\medskip

{\it Step 4: For any $\beta \in \IR$, $p \in [1, \infty)$, there exists $r_0 \geq 0$ such that if 
$r \geq r_0$, $g \in \mathsf{B}^\beta_{p,p}(\IR^m, \IC^\ell)$, then (\ref{eq1}) has a unique solution 
$u\in\mathsf{B}^{\beta+n}_{p,p}(\IR^m, \IC^\ell)$. }

\medskip

The uniqueness follows from Step 3. We show the existence by a duality argument. We think of 
$r^n \mathrm{e}^{\mathrm{i}\theta_0} - Q$ as an operator from 
$\mathsf{B}^{\beta+n}_{p,p}(\IR^m, \IC^\ell)$ to $\mathsf{B}^{\beta}_{p,p}(\IR^m, \IC^l)$. 
By Step 3, if $r$ is sufficiently large, its range is a closed subspace of  
$\mathsf{B}^{\beta}_{p,p}(\IR^m, \IC^\ell)$. Assume that it does not coincide with the whole space. 
Then, applying a well known consequence of the theorem of 
Hahn-Banach and (v), there exists $h \in \mathsf{B}^{-\beta}_{p',p'}(\IR^m, \IC^l)$, $h \neq 0$, such that 
$$
\left\langle(r^n e^{i\theta_0} - Q)u, \overline h\right\rangle = 0
\>\text{ for all } u \in \mathsf{B}^{\beta+n}_{p,p}(\IR^m, \IC^\ell). 
$$
This implies that 
\begin{equation}\label{eq13}
(r^n \mathrm{e}^{-\mathrm{i}\theta_0} - Q^\ast) h = 0.
\end{equation}
Now, it is easily seen that $Q^\ast$ satisfies the assumptions of Proposition \ref{pr2} if we 
replace $\theta_0$ with $-\theta_0$. We deduce from Step 3 that, if $r$ is sufficiently large, 
(\ref{eq13}) implies $h = 0$, a contradiction. 

\medskip

{\it Step 5: For any $\beta \in \IR$ there exists $r_0 \geq 0$ such that if $r \geq r_0$, 
$g \in \mathsf{B}^\beta_{\infty,\infty}(\IR^m, \IC^\ell)$, then 
(\ref{eq1}) has a unique solution $u\in\mathsf{B}^{\beta+n}_{\infty,\infty}(\IR^m, \IC^\ell)$. }

\medskip

In the proof of Lemma 2.4 from \cite{davide} it is shown that for any 
$g \in \mathsf{B}^\beta_{\infty,\infty}(\IR^m)$, 
there is a sequence $(g_k)_{k \in \IN}$ in $\mathsf{S}(\IR^m)$ converging to $g$ in $\mathsf{S}'(\IR^m)$ 
and bounded in $\mathsf{B}^\beta_{\infty,\infty}(\IR^m)$. So we take a sequence $(g_k)_{k \in \IN}$ in 
$\mathsf{S}(\IR^m, \IC^\ell)$ converging to $g$ in $\mathsf{S}'(\IR^m, \IC^\ell)$ and bounded in 
$\mathsf{B}^\beta_{\infty,\infty}(\IR^m, \IC^\ell)$. We fix $\gamma$  larger than  $\beta + \frac{m}{2}$ 
and think of $g_k$ as an element of $\mathsf{B}^\gamma_{2,2}(\IR^m, \IC^\ell)$. Then, by Step 4, if $r$ 
is sufficiently large, the equation
$$
r^n \mathrm{e}^{\mathrm{i}\theta_0} u_k - Q u_k = g_k
$$
has a unique solution $u_k$ in $\mathsf{B}^{\gamma+n}_{2,2}(\IR^m, \IC^\ell)$. By (iii), 
$u_k \in \mathsf{B}^{\beta+n}_{\infty,\infty}(\IR^m, \IC^\ell)$ and, by Step 3, if $r$ is sufficiently 
large, the sequence $(u_k)_{k \in \IN}$ is bounded in $\mathsf{B}^{\beta+n}_{\infty,\infty}(\IR^m, \IC^\ell)$, 
because $(g_k)_{k \in \IN}$ is bounded in $\mathsf{B}^{\beta}_{\infty,\infty}(\IR^m, \IC^\ell)$. Then, by (v) 
and the theorem of Alaoglu, we may assume, possibly passing to a subsequence,  that there exists 
$u\in\mathsf{B}^{\beta+n}_{\infty,\infty}(\IR^m, \IC^\ell)$ such that  
$$
\lim_{k \to \infty} u_k = u\>\text{ in the weak topology  }\>\ 
w(\mathsf{B}^{\beta+n}_{\infty,\infty}(\IR^m, \IC^\ell), 
\mathsf{B}^{-\beta-n}_{1,1}(\IR^m, \IC^\ell)).
$$
Such convergence implies convergence in $\mathsf{S}'(\IR^m, \IC^\ell)$. So 
$$
(r^n \mathrm{e}^{\mathrm{i}\theta_0} - Q) u_k \to (r^n \mathrm{e}^{\mathrm{i}\theta_0} - Q) u
\text{ as $k\to\infty$ in $\mathsf{S}'(\IR^m, \IC^\ell)$}. 
$$
We deduce that $(r^n \mathrm{e}^{\mathrm{i}\theta_0} - Q) u = g$. 

\medskip

{\it Step 6: Full statement.}

\medskip

This is a simple consequence of Step 4, Step 5 and the interpolation property (iv).

\end{proof}

\section*{Acknowledgements}
B.G. has been financially supported by the SFB~647 ``Space-Time-Matter''. D.P. is member of Italian 
CNR-GNAMPA and has been partially supported by PRIN 2010 
M.I.U.R. ``Problemi differenziali di evoluzione: approcci deterministici e stocastici e loro interazioni".

\end{document}